\documentclass[a4paper, 11pt]{amsart}
\usepackage[mathscr]{eucal}
\usepackage{amssymb, amsfonts, amsmath, amsthm, graphicx, cases, color, tikz} 

\newtheorem{proposition}{Proposition}[section]
\newtheorem{theorem}{Theorem}[section]
\newtheorem{lemma}[proposition]{Lemma}

\theoremstyle{definition}
\newtheorem{definition}{Definition}[section]
\newtheorem{assumption}{Assumption}[section]
\newtheorem*{thank}{Acknowledgments}
\theoremstyle{remark}
\newtheorem{remark}{Remark}[section]

\numberwithin{equation}{section}

\title{Stability of standing waves for all frequencies to nonlinear Schr\"odinger
equations with potentials in one dimension} 
\author{Noriyoshi Fukaya \and 
Masahiro Ikeda \and 
Hiroaki Kikuchi}

\dedicatory{Dedicated to Professor Yoshio Tsutsumi on the occasion of his 70th birthday.}

\date{\today}
\newcommand{\R}{\mathbb{R}}

\newcommand{\e}{\varepsilon}
\newcommand{\p}{\partial}

\renewcommand{\Re}{\operatorname{Re}}
\renewcommand{\Im}{\operatorname{Im}}



\begin{document}
\maketitle

\begin{abstract}
In this paper, we study the orbital stability of standing waves for one-dimensional nonlinear Schr\"odinger equations with potentials.
We show that the standing waves are orbitally
stable for all frequencies in the $L^{2}$-
subcritical and critical cases.
Since the presence of potentials breaks the 
scale invariance of the equations, it is a 
delicate problem to apply the abstract theory of Grillakis, Shatah, and Strauss (1987)
directly without a perturbative argument.
For this reason, little is known about the 
orbital stability of standing waves for \textit{all} 
frequencies in the non-scale-invariant setting.
We overcome this difficulty by employing the 
approach of Noris, Tavares, and 
Verzini (2014).
\end{abstract}

\section{Introduction}
In this paper, we consider the following 
one-dimensional nonlinear Schr\"{o}dinger equations with a potential:
\begin{equation}\label{NLS1}
    i \psi_{t}
    =-\psi_{xx} 
    + V(x) \psi -
    |\psi|^{p-1}\psi,
    \quad (t, x)\in \R\times \R, 
\end{equation}
where $1 < p < \infty$, $V$ is a real-valued function, and $\psi=\psi(t, x)$ is the complex-valued unknown function.

Equation~\eqref{NLS1} is known as a model proposed to describe
the local dynamics at a nucleation site. 
Here, the potential $V$ simulates 
a local depression in the ion density 
(see Rose and Weinstein~\cite{MR939275}).
In addition, \eqref{NLS1} also appears 
in the Bose-Einstein condensate. 
The following 
three-dimensional nonlinear Schr\"{o}nger equation 
\begin{equation}\label{BEC}
    i \psi_t
    =-\Delta \psi 
    + |x|^{2} \psi 
    - |\psi|^{2}\psi\quad \mbox{in $\R \times \R^{3}$}
    \end{equation}
describes the Bose-Einstein condensate with attractive 
inter-particle interactions under a magnetic trap. 
Equation~\eqref{NLS1} is derived from \eqref{BEC}
and seems appropriate when its longitudinal dimension is much longer than its transverse ones
(see \cite{Carr_2000} for the derivation of \eqref{NLS1}).

Now, we state our assumption on the 
potential $V$: 

\begin{enumerate}
\item[(V1)]  $V\in L_{\text{loc}}^1(\R; \R)\cap C^{1}(\R \setminus \{0\})$ is even, non-decreasing function of $r = |x|$, and  $\p_{x} V \not\equiv 0$.
\end{enumerate}
Put 
\begin{equation} \label{eq1-2}
	\omega_{1} 
    := - \inf\Bigl\{
	\int_{\R} (|u_{x}|^{2} + V(x) |u|^{2}) dx
	\colon u \in H^{1}(\R)
    \text{ with }
    V|u|^{2}\in L^1(\R), \:
    \|u\|_{L^{2}} = 1
	\Bigr\}. 
\end{equation}
Under the assumption (V1), we see that 
$\omega_{1}$ is well-defined, that is, $\omega_{1} < \infty$ 
(see Lemma~\ref{lema-1} (i) below).    

\begin{definition}
Let $\omega > \omega_{1}$. 
We define a real Hilbert space $X$ by 
\[	X := \{ 
	u \in H^{1}(\R) \colon
    V(x) |u|^{2} \in L^1(\R)
	\} \]		
equipped with the inner product
\[	(u, v)_{X_{\omega}} := 
    \Re \int_{\R} 
    \bigl(u_x(x)\overline{v_x(x)} + V(x) u(x) 
    \overline{v(x)} + \omega u(x) \overline{v(x)}\bigr) dx.	\]   
The norm of $X$ is denoted by $\|\cdot\|_{X_{\omega}}$. In particular, we use $\|\cdot\|_X:=\|\cdot\|_{X_{\omega_1+1}}$. 
In addition, we denote by $X^{*}$ the dual of 
the function space $X$. 
\end{definition}
We also find from the assumption 
(V1) that 
\begin{equation} \label{eq1-1}
    X \hookrightarrow H^{1}(\R). 
\end{equation}
See Lemma \ref{lema-1} (ii) below.

Throughout this paper, we always assume the following: 
\begin{assumption} 
\label{l-wellposed}
For any $\psi_{0} \in X$, there exist 
$T_{\max} = T_{\max}(\|\psi_{0}\|_X) > 0$ 
and a solution $\psi \in C([0, T_{\max}), 
X)$ to \eqref{NLS} with $\psi|_{t = 0} 
= \psi_{0}$ satisfying 
\[  E(\psi(t)) = E(\psi_{0}), 
    \quad 
    \|\psi(t)\|_{L^{2}} = \|\psi_{0}\|_{L^{2}} 
    \quad \mbox{for all $t \in [0, 
    T_{\max})$},  
    \]
where $E$ is the energy defined by 
\[  E(u) = \frac{1}{2} 
    \int_{\R} 
    (|u_{x}(x)|^{2} 
    + V(x) |u(x)|^{2}) dx
    - \frac{1}{p + 1} \|u\|_{L^{p + 1}}^{p + 1}.    \]
\end{assumption}

\begin{remark}
The assumption \ref{l-wellposed} holds if there exist 
real valued functions 
$V_{1}$ and $V_{2}$ such that 
$V(x) = V_{1}(x) + V_{2}(x)$, 
where $V_{1}$ and $V_{2}$ satisfy
 the following: 
 \begin{enumerate}
 \item[\textrm{(i)}] 
 $V_{1} \in C^{\infty} (\R)$, 
 $V_{1}(x) \geq 0$ 
 for $x \in \R$, 
$\p_{x}^{m} 
V_{1} \in L^{\infty}(\R)$ 
for $m \geq 2$
\item[\textrm{(ii)}]
$V_{2} \in L^{1}(\R) + L^{\infty} 
(\R)$. 
 \end{enumerate}
(see \cite[Theorem 3.3.1,  
3.3.5, 3.3.9, 
Proposition 4.2.3 and 
Theorem 9.2.6]
{MR2002047} and \cite{MR1016082}). 
\end{remark}

Here, we are interested in a standing wave 
solution to \eqref{NLS}. 
By a \textit{standing wave}, we mean a solution to 
\eqref{NLS1} of the form 
    \[
    \psi(t, x) = e^{i \omega t} \phi_{\omega}(x)
    \qquad (\omega \in \R). 
    \]
Then we see that the function $\phi_{\omega}$ 
satisfies the following elliptic equations:
\begin{equation}
    \label{sp1}
    - \phi_{xx} 
    + V(x) \phi + \omega \phi
    - |\phi|^{p - 1} \phi = 0, 
    \quad x\in\R. 
\end{equation}
We define the action functional 
$S_{\omega} \in C^{1}(X, \R)$ by
\[  S_{\omega}(u) 
    := E(u) + \frac{\omega}{2} 
    \|u\|_{L^{2}}^{2}.  \]
Then we see that $u \in X$ is a weak solution to \eqref{sp1} if and only if 
$S_{\omega}^{\prime}(u) = 0$ in $X^{*}$. 
We call by \textit{ground state} a solution to 
\eqref{sp1} which minimizes the action functional $S_{\omega}$ among all non-trivial solutions to \eqref{sp1} in $X$. 
We denote by $\mathcal{A}_{\omega}$ the set of all non-trivial solutions to \eqref{sp1} in $X$ and by $\mathcal{G}_{\omega}$ that of all ground states to \eqref{sp1}. Namely, 
\[  \mathcal{A}_{\omega} 
    := \{
    u \in X \setminus \{0\} 
    \colon S_{\omega}'(u) 
    = 0 \mbox{ in }X^{*}\}
    \]
and  
\begin{equation} \label{eq1-3}
    \mathcal{G}_{\omega} 
    := \{
    u \in \mathcal{A}_{\omega} 
    \colon S_{\omega}(u) 
    \leq S_{\omega}(v) 
    \mbox{ for all }
    v \in \mathcal{A}_{\omega}\}. 
\end{equation}
The following properties of ground states are well-known. For the readers' convenience, we give these proofs later.
\begin{theorem} \label{thm:1.1}
Let $\omega > \omega_{1}$ and 
$1 < p < \infty$. Assume that (V1).
Then there exists $(\phi_\omega)_{\omega>\omega_1}$ such that the following is true.
\begin{enumerate}
\renewcommand{\labelenumi}{(\roman{enumi})}
\item $\phi_{\omega}\in \mathcal{A}_\omega$ is positive, even, and decreasing in $r = |x|$.
\item $\mathcal{G}_\omega=\{e^{i\theta} \phi_{\omega}\colon \theta\in\R\}$. 
\item $S_\omega''(\phi_{\omega})$ has a unique negative eigenvalue $\nu_\omega$, which is simple. 
\item $\ker S_\omega''(\phi_{\omega})=\operatorname{span}\{i\phi_\omega\}$. 
\item $\sigma(S_\omega''(\phi_{\omega}))\setminus\{\nu_\omega, 0\}\subset[\delta, \infty)$ for some $\delta>0$, 
where $\sigma(S_\omega''(\phi_{\omega}))$ 
is the set of spectrum of the operator 
$S_{\omega}''(\phi_{\omega})$. 
\end{enumerate}
\end{theorem}    
Next, we study the orbital stability 
of standing waves $e^{i \omega t} \phi_{\omega}$. 
The orbital stability/instability of standing 
waves is defined by 
the following: 
\begin{definition}
We say that the standing wave $e^{i \omega t} \phi_{\omega}$ 
is orbitally stable if for any 
$\e > 0$, there exists $\delta > 0$ such that if 
\[  \|\psi_{0} -\phi_{\omega}\|_{X} < \delta,  \]
then the corresponding solution 
$\psi$ to \eqref{NLS1} with $\psi|_{t = 0} 
= \psi_{0}$ exists globally and satisfies 
\[  \sup_{t > 0} \inf_{\theta \in \R} \|\psi(t) - e^{i \theta} \phi_{\omega}\|_{X}  
    < \e.   \]
Otherwise, we say that the standing wave
$e^{i \omega t} \phi_{\omega}$ is orbitally unstable 
\end{definition}
The orbital stability and instability of the 
standing waves to nonlinear 
Schr\"{o}dinger equations with potentials 
has been studied so far 
(see \cite{MR3638314, MR646873, MR677997, 
MR4020633, MR1973275, 
MR1948875, MR2379460, MR1989539, 
MR3842879, MR3792709, MR939275, MR2239285, 
MR691044} and references therein). 
Let us recall some of them. 
To this end, we consider the following 
nonlinear Schr\"{o}dinger equations in general dimensions: 
\begin{equation}\label{NLS}
    i \psi_t 
    =-\Delta \psi 
    + V(x) \psi 
    - |\psi|^{p-1}\psi \quad 
    \text{in }\R \times \R^{d}, 
\end{equation}
where 
$d \ge 1$, $1 < p < 2^{*} - 1$, and 
$V$ is a real valued function. 
Here 
\[  2^{*} =
    \left\{\begin{alignedat}{2}
   & \infty & \quad 
   &\text{when }d = 1, 2,
\\ &\frac{2d}{d - 2} & \quad 
   &\text{when }d \ge 3.
    \end{alignedat}\right.  \]
When we consider the standing waves 
$\psi(t, x) = e^{i \omega t} \phi_{\omega}(x)$, 
then $\phi_{\omega}$ satisfies 
\begin{equation}\label{sp}
    - \Delta \phi 
    + V(x) \phi 
    + \omega \phi 
    - |\phi |^{p - 1} \phi 
    = 0 \quad \text{in }\R^{d}. 
\end{equation}

We first consider the case where $V \equiv 0$. 
In this case, 
the equation \eqref{NLS} is invariant under the 
scaling 
\begin{equation} \label{eq-scale}
    \psi(t, x) 
    \mapsto \psi_{\lambda}(t, x) 
    = \lambda^{\frac{2}{p-1}} \psi(\lambda^{2} t, \lambda x). 
\end{equation}
Note that 
\begin{equation} \label{eq-scale2}
    \|\psi_{\lambda}(0, \cdot)\|_{L^{2}} 
    = \lambda^{\frac{2}{p - 1} - \frac{d}{2}}
    \|\psi(0, \cdot)\|_{L^{2}}.  
\end{equation}
Thus, we see that the scaling \eqref{eq-scale} preserves 
the $L^{2}$-norm $\|\cdot\|_{L^{2}}$ 
when $p = 1 + 4/d$.  
For this reason, the exponent 
$p = 1 + 4/d$
is referred to as ``$L^{2}$-critical''. 
Let $\phi_{\omega}$ be the ground state 
of \eqref{sp}.  
Cazenave and Lions~\cite{MR677997} proved that standing waves $e^{i \omega t} \phi_{\omega}$
are stable 
when $L^{2}$-subcritical case 
$1 < p < 1 + 4/d$.   
On the other hand, 
Berestycki and Cazenave~\cite{MR646873} showed that standing waves 
$e^{i \omega t} \phi_{\omega}$
are unstable 
when $L^{2}$-supercritical case
$1 + 4/d < p < 2^{*} 
- 1$. 
Concerning with the $L^{2}$-critical case 
$p = 1 + 4/d$, 
Weinstein~\cite{MR691044} proved that the standing waves 
$e^{i \omega t} \phi_{\omega}$ 
are unstable for any $\omega > 0$, 
where $\phi_{\omega} \in H^{1}(\R^{d})$ 
is any non-trivial solution to \eqref{sp} (not 
necessary the ground state). 
When we consider the ground state 
$\phi_{\omega} $ to \eqref{sp}, we can also 
study the stability of standing waves 
$e^{i \omega t} \phi_{\omega}$
by applying the 
abstract theory of Grillakis, Shatah 
and Strauss~\cite{MR901236}, 
in which they showed that the standing wave $e^{i \omega t} \phi_{\omega}$ is
stable if 
$\frac{d}{d \omega} \|\phi_{\omega}\|_{L^{2}}^{2} > 0$ 
and unstable if $\frac{d}{d \omega} \|\phi_{\omega}\|_{L^{2}}^{2}
< 0$ 
under several assumptions.
See Comech and Pelinovsky~\cite{MR1995870}, Ohta~\cite{MR2785894}, and Maeda~\cite{MR2923422} for the case of $\frac{d}{d \omega} \|\phi_{\omega}\|_{L^{2}}^{2} = 0$. 
When $V \equiv 0$, it follows from 
\eqref{eq-scale2} that
\begin{equation} \label{eq-sca}
    \|\phi_{\omega}\|_{L^{2}}^{2} 
    = \omega^{\frac{2}{p - 1} 
    - \frac{d}{2}} 
    \|\phi_{1}\|_{L^{2}}^{2}. 
\end{equation}
From this, we can check the sign of 
$\frac{d}{d \omega} \|\phi_{\omega}\|_{L^{2}}^{2}$
and find that 
\[  \frac{d}{d \omega} \|\phi_{\omega}\|_{L^{2}}^{2}
    \left\{\begin{alignedat}{2}
   &> 0 & \quad 
   &\text{if }1 < p < 1 + \frac{4}{d}, 
\\ &= 0 & \quad 
   &\text{if }p = 1 + \frac{4}{d}, 
\\ &< 0 & \quad 
   &\text{if }1 + \frac{4}{d} < p < 2^{*} - 1. 
    \end{alignedat}\right.  \]
It is well-known that the 
ground state $\phi_{\omega}$ 
satisfies the assumptions of 
the abstract theory of 
Grillakis, Shatah and Strauss
~\cite{MR901236}. 
Thus, we can find that 
the standing wave $e^{i \omega t} 
\phi_{\omega}$ is stable for 
$1 < p < 1 + 4/d$ and 
unstable for $1 + 4/d < p < 2^{*} - 1$. 

Now we pay our attention to the case 
of $V \not\equiv 0$. 
In this case, the equations \eqref{NLS} 
are not scale invariant, so that \eqref{eq-sca} does 
not hold. 
For this reason, we cannot check the sign 
of $\frac{d}{d\omega}\|\phi_{\omega}\|_{L^{2}}^{2}$ at least directly. 
Rose and Weinstein~\cite{MR939275} 
employed the bifurcation theory and 
computed the sign of 
$\frac{d}{d\omega}\|\phi_{\omega}\|_{L^{2}}^{2}$ 
for $\omega \in \R$ which is close to $\omega_{1}$. 
After that, 
Fukuizumi and Ohta~\cite{MR1948875} gave another sufficient condition
of stability
which is based on the result of Grillakis, Shatah, and Strauss~\cite{MR901236}
and showed that standing waves
$e^{i \omega t} \phi_{\omega}$
$(\phi_{\omega} \in \mathcal{G}_{\omega})$ 
are stable
for sufficiently large $\omega > 0$ 
when $1 < p < 1 + 4/d$. 
Fukuizumi~\cite{MR2123132} also showed the stability for large $\omega$ when $p=1+4/d$. 
Since they used a kind of perturbation argument 
to verify the sufficient condition, 
some restrictions on the frequency $\omega \in \R$ were required. 
There are several results in which  
the stability of standing waves 
$e^{i \omega t} \phi_{\omega}$ is obtained under 
some conditions on the frequency $\omega \in \R$.
Compared with these, 
there are few results in which 
the stability of standing waves 
are obtained without any additional 
conditions on the frequency $\omega \in \R$. 
McLeod, Stuart, and Troy~\cite{MR1989539} studied the behavior of 
$\p_{\omega} \phi_{\omega}$ using the continuity argument and obtained the stability of standing waves 
in one dimensional case $d = 1$. 
In particular, the standing waves 
$e^{i \omega t} \phi_{\omega}$ are 
stable for any 
$\omega > \omega_{1}$ 
if $1 < p \le 5$.
However, they assumed the boundedness 
on the potential $V$. 
Thus, the harmonic potential 
$V(x) = |x|^{2}$ and 
the inverse potential 
$V(x) = - |x|^{-\theta}$ $(\theta \in (0, 1))$ 
were not treated in their result. 

Here, using the argument
of \cite{MR3318740}, we study
the stability of standing waves
$e^{i \omega t} \phi_{\omega}$
$(\phi_{\omega} \in \mathcal{G}_{\omega})$ 
for all $\omega > \omega_{1}$. 
Our result includes unbounded potentials
such as $V(x) = |x|^{2}, - |x|^{-\theta}\; 
(\theta \in (0, 1))$ and so on. 
To study the orbital stability of the 
standing waves $e^{i \omega t} \phi_{\omega}$, 
we need the following additional condition:
\begin{enumerate}
\item[(V2)] $2V+xV'$ is a non-decreasing function of $r=|x|$. Furthermore, $2V+xV'$ is not a constant function if $p=5=1+4/d|_{d=1}$.
\end{enumerate}

\begin{remark}
Concerning the assumption (V2), note first that in the one-dimensional setting the case 
where $2V+xV'$ is constant is already excluded by (V1). Indeed, in this case we can write 
$V(r) = c_1 + c_2 r^{-2}$ for some $c_1 \in \R$ and $c_2 \in \R$. If $c_2 \ne 0$, then 
$V \notin L_{\mathrm{loc}}^1(\R)$; whereas if $c_2 = 0$, then $\partial_x V \equiv 0$.

In higher dimensions, however, one may consider the case where $2V+xV'$ is constant, 
in particular $V(r) = -r^{-2}$. In this situation, equation \eqref{NLS} is invariant 
under the scaling \eqref{eq-scale}. Hence, when $p=1+4/d$, it follows from the same 
argument as in \cite{MR691044} that all standing waves are strongly unstable. 
Therefore, in the assumption (V2), it is essential to exclude the case where 
$2V+xV'$ is constant for $p = 1 + 4/d$.
\end{remark}

Then we obtain the following: 
\begin{theorem}\label{thm-stability}
Let $d = 1, 
1 < p \le 5$, and 
$\phi_{\omega} \in \mathcal{G}
_{\omega}$.
Assume that Assumption \ref{l-wellposed}, 
(V1), and (V2). 
Then the standing wave $e^{i \omega t}\phi_{\omega}$ is stable for any $\omega > \omega_{1}$. 
\end{theorem}

\begin{remark}
\begin{enumerate}
\renewcommand{\labelenumi}{(\roman{enumi})}
\item Our result covers $V(x) = |x|^{2}$, $V(x) = -|x|^{-\theta}$ with $\theta \in (0,1)$, and combinations of these.
As far as the authors know, 
in these cases, the stability of 
standing waves for all frequencies has not been studied so far. 
\item For the case of $1 + 4/d
< p < 2^{*} - 1$, it is known that 
the standing waves $e^{i \omega t} 
\phi_{\omega}$ $(\phi_{\omega} \in 
\mathcal{G}_{\omega})$ are unstable 
for sufficiently large $\omega > 0$. 
Thus, we cannot obtain the 
stability for all frequency in this 
case, so that 
the condition $1 < p \leq 1 + 
4/d$ in Theorem~\ref{thm-stability} is essential. 
\end{enumerate}    
\end{remark}
To prove Theorem~\ref{thm-stability}, 
we use the argument of \cite{MR3318740}, 
in which they studied the stability for the 
following nonlinear Schr\"{o}dinger 
equation on the unit ball $B_1$ 
with the Dirichlet boundary condition: 
\begin{equation}\label{NLS-b}
    \left\{\begin{alignedat}{2}
   &i \psi_t
     =- \Delta \psi 
     - |\psi|^{p - 1} \psi& \quad
   &\text{in }\R \times B_{1}, 
\\ & \psi = 0 & \quad 
   &\mbox{in }\R \times \p B_{1}. 
    \end{alignedat}\right.
\end{equation}
Actually, since the equations 
\eqref{NLS-b} are not scale-invariant, 
we cannot check the sign of 
$\frac{d}{d \omega}\|\phi_{\omega}\|_{L^{2}}^{2}$
at least directly. 
To overcome the difficulty, 
Noris, Tavares, and Verzini~\cite{MR3318740} considered
the following optimization problem with two constraints: 
\[  M_{\alpha} := 
    \sup\Bigl\{
    \int_{B_{1}} |u(x)|^{p + 1} dx \colon 
    u \in H_{0}^{1}(B_{1}), 
    \: 
    \int_{B_{1}} 
    |u(x)|^{2} dx = 1, \: 
    \int_{B_{1}} 
    |\nabla u(x)|^{2} dx 
    = \alpha
    \Bigr\} \]
for $\alpha > 0$. 
Then they showed that for any $\alpha > \lambda_{1}(B_{1})$, where $\lambda_{1}(B_{1})$ is the first eigenvalue of the Dirichlet Laplacian in the unit ball $B_{1}$, there exists a maximizer $u_{\alpha} \in H_{0}^{1}(B_{1})$ for $M_{\alpha}$. 
In addition, 
$u_{\alpha} \in 
H_{0}^{1}(B_{1})$ is a positive function
and 
there exist 
$\mu_{\alpha} > 0$ and 
$\omega_{\alpha} 
> - \lambda_{1}(B_{1})$
such that 
\begin{align*}
    - \Delta u_{\alpha} 
    + \omega_{\alpha} 
    u_{\alpha} 
   &=\mu_{\alpha} u_{\alpha}^{p}, &
    \int_{B_{1}} 
    |u_{\alpha}(x)|^{2} dx 
   &= 1, &
    \int_{B_{1}} 
    |\nabla u_{\alpha}(x)|^{2} dx 
   & = \alpha. 
\end{align*}
Then they showed that 
\[  \frac{d}{d\omega}\|\phi_\omega\|_{L^{2}}
    ^{2} = C_{\alpha} 
    \frac{\p \mu_{\alpha}}
    {\p \alpha},    \]
where $C_{\alpha}$ is 
a positive constant 
depending on $\alpha 
> \lambda_{1}(B_{1})$. 
Thus, 
in order to check the sign 
of $\frac{d}{d\omega}\|\phi_\omega\|_{L^{2}}^2$, 
it suffices to do that of 
$\p_{\alpha} \mu_{\alpha}$. 
Then by exploiting the 
Dirichlet boundary condition, 
they showed that 
$\p_{\alpha} \mu_{\alpha} > 0$ for all 
$\alpha > \lambda_{1}(B_{1})$ when $1 < p \leq 1 + 4/d$. 

Here, we simplify the above argument and determine the sign of $\frac{d}{d\omega}\|\phi_\omega\|_{L^{2}}^2$ directly.
We also note that since our spatial domain is the entire space, we cannot apply the argument of \cite{MR3318740} directly.
To overcome this difficulty, we investigate the asymptotic behavior of $\p_{\omega} \phi_{\omega}$
more precisely and use (V2).
\begin{remark}
Here, we emphasize that 
we essentially need the condition $d = 1$ 
for obtaining 
$\p_{\omega} \phi_{\omega}(0) > 0$
(see Lemma \ref{lem:v0posi} 
below for details). Once 
$\p_{\omega} \phi_{\omega}(0) > 0$ 
is established, our result can be extended to higher dimensions under certain assumptions.
\end{remark}

The rest of this paper is organized as follows.
In Section~\ref{sec-ex}, we show the existence of ground states. 
In Section~\ref{sec-po-gs}, we prove the positivity and evenness of the ground states. 
In Section~\ref{sec-unique}, we show the uniqueness of the ground state. 
In Section~\ref{sec-nond}, we establish its non-degeneracy. 
In Section~\ref{sec-sta}, we study the orbital stability of the standing waves $e^{i \omega t}\phi_{\omega}$ and give the proof of Theorem~\ref{thm-stability}. 
In Appendix~\ref{sec-bb} we prove the finiteness of $\omega_{1}$ and \eqref{eq1-1}. 
Finally, in Appendix~\ref{sec-Morse}, we prove an auxiliary result that is needed for the non-degeneracy of the ground state.

\section{Existence of ground state}\label{sec-ex}
In this section, we show the existence 
of the ground state of \eqref{sp1}. 
To this end, we consider the following 
minimization problem:
\begin{equation} \label{eq2-8}
    d(\omega) = \inf\{
    S_{\omega}(u) \colon u \in X \setminus 
    \{0\}, \: K_{\omega}(u) = 0\},  
\end{equation}
where 
\begin{equation}\label{eq-N}
    K_{\omega}(u) 
    := \langle S_{\omega}'(u), u \rangle 
    = \|u\|_{X_{\omega}}^2 
    - \|u\|_{L^{p+1}}^{p + 1}
\end{equation}
is the Nehari functional. Concerning the minimization problem, 
we obtain the following. 
\begin{theorem}\label{thm-ex}
Assume that (V1).
There exists a minimizer
$\phi_{\omega} \in X \setminus \{0\}$ for $d(\omega)$. 
\end{theorem}


We put 
\begin{align} 
    \label{eq-I}
& Q_{\omega}(u) = S_{\omega}(u) - 
	\frac{1}{p+1} K_{\omega}(u) 
	= \frac{p - 1}{2(p+1)}
	\|u\|_{X_{\omega}}^{2}, \\ 
\label{eq-J}
& N(u) = S_{\omega}(u) - 
	\frac{1}{2} K_{\omega}(u) 
	= \frac{p - 1}{2(p+1)}
	\|u\|_{L^{p+1}}^{p+1}.  
\end{align}
Then we have
\begin{equation} \label{eq2-4}
\begin{split}
	d(\omega)
	&= \inf\{ Q_{\omega}(u) 
\colon u \in X 
\setminus \{0\}, \: K_{\omega} (u) = 0 \} \\
    & = \inf\{N(u) 
\colon u \in X 
\setminus \{0\}, \: K_{\omega} (u) = 0 \}.
\end{split}
\end{equation}

\begin{lemma}\label{lemma-lb}
$d(\omega) > 0$ for $\omega>\omega_1$. 
\end{lemma}

\begin{proof}
Let $u \in X \setminus \{0\}$ satisfy $K_{\omega}(u) = 0$. By the Sobolev embedding \eqref{eq1-1}, we have 
\[  \|u\|_{X_{\omega}}^{2} = \|u\|_{L^{p+1}}^{p+1} 
	\leq C\|u\|_{X_{\omega}}^{p+1} \]
for some $C > 0$, which is independent of $u \in X$. This implies that $\|u\|_{X_{\omega}}\geq (1/C)^{\frac{1}{p-1}}$. 
From this and \eqref{eq2-4}, we have obtained the desired result. 
\end{proof}

\begin{lemma}\label{lem:2.2}
If $u\in X$ satisfies $Q_\omega(u)\le d(\omega)\le N(u)$, then $u$ is a minimizer for $d(\omega)$. 
\end{lemma}

\begin{proof}
Note that $u\ne 0$ because $N(u)\ge d(\omega)>0$. We put $\lambda_0:= 
(Q_\omega(u)N(u)^{-1})^{\frac{1}{p-1}}
\in (0, 1]$. Then we have $K_\omega(\lambda_0u)=0$. From 
\eqref{eq2-4} and our assumption, we obtain 
\[  Q_\omega(u)
    \le d(\omega)
    \le Q_{\omega}(\lambda_0 u)
    =\lambda_0^2 Q_{\omega}(u)
    \le Q_{\omega}(u).   \]
This implies $Q_\omega(u)=d(\omega)$, $\lambda_0=1$, and $K_\omega(u)=0$. 
\end{proof}


\begin{proof}[Proof of Theorem \ref{thm-ex}]
Let $\{u_{n}\}$ be a minimizing sequence of $d(\omega)$, that is, $K_\omega(u_n)=0$ and
\begin{align*}
    Q_\omega(u_n)
    =N(u_n)\to d(\omega) 
    \quad \mbox{as } n \to \infty.
\end{align*}
Note that $\{u_{n}\}$ is bounded in $X$. We put $v_{n} := |u_{n}|^{*}$, where $f^{*}$ is the Schwarz rearrangement of 
the function $f$. 
Then we see that $\|v_n\|_{L^{q}}=\|u_n\|_{L^{q}}$ for any $q\ge 1$ and $\|\nabla v_n\|_{L^2}\le\|\nabla u_n\|_{L^2}$ (see \cite[Section~3.3 and Lemma~7.17]{LiebLoss2001}). Moreover, for any $f\in X$, the inequality 
\begin{equation}\label{eq:Vreaineq}
    \int_{\R} V(x)|f^*|^2 dx \le \int_{\R} V(x)|f|^2 dx
\end{equation}
holds. Indeed, by the increasing monotonicity 
of $V$, we have $\int_{\R}V(x)\chi_{\{x\colon 
|f(x)|^2>s\}} \, ds 
\ge\int_{\R}V(x)\chi_{\{x\colon |f^*(x)|^2>s\}} 
\, ds$ for all $s>0$. From this and the layer cake representation $|f(x)|^2=\int_0^\infty\chi_{\{x\colon |f(x)|^2>s\}}\,ds$ (see \cite[Theorem 1.13 (4)]{LiebLoss2001}), we obtain
\begin{align*}
    \int_{\R}V(x)|f^*|^2 \,dx
   &=\int_{\R}\int_0^\infty V(x)\chi_{\{x\colon |f^*(x)|^2>s\}}\,ds\,dx
\\ &\le \int_0^\infty\int_{\R}V(x)\chi_{\{x\colon |f(x)|^2>s\}}\,ds\,dx
    =\int_{\R}V(x)|f|^2 \,dx.
\end{align*}
Therefore, we have $\|v_n\|_X\le \|u_n\|_X$. This implies $\{v_{n}\}$ is also bounded in $X$. Then up to subsequence (we still denote it by the same symbol), there exists $u_{\infty} \in X$ such that $v_{n}$ converges weakly to $u_{\infty}$ in $X$. By the weak lower semi-continuity of norms, we have \begin{align*}
    Q_\omega(u_\infty)\le \lim_{n\to\infty}Q_\omega(v_n)
    \le \lim_{n\to\infty}Q_\omega(u_n)
    =d(\omega).
\end{align*}
Moreover, by the radial monotonicity compactness lemma \cite[Proposition 1.7.1]{MR2002047},
we have  
\begin{equation} \label{eq2-7}
    N(v_\infty)
    =\lim_{n\to\infty}N(v_{n})
    =\lim_{n\to\infty}N(u_{n}) 
    =d(\omega). 
  \end{equation}
By Lemma~\ref{lem:2.2}, $u_{\infty} \in X$ is a minimizer for $d(\omega)$. This completes the proof. 
\end{proof}
We set 
\[  \mathcal{M}_\omega
    :=\{u\in X\colon K_\omega(u)=0,\: S_\omega(u)=d(\omega)\}, \]
which is the set of all minimizers for $d(\omega)$.

\begin{theorem}\label{thm:2.2}
For any $\omega > 0$, we have 
$\mathcal{M}_\omega=\mathcal{G}_\omega$. 
\end{theorem}
\begin{proof}
Let $\phi_{\omega} \in \mathcal{G}_{\omega}$. 
Then we have $K_{\omega}(\phi_{\omega}) 
= \langle S_{\omega}'(\phi_{\omega}), 
\phi_{\omega} \rangle = 0$. 
This together with \eqref{eq2-8} yields that 
    \begin{equation} \label{eq2-9}
    d(\omega) \leq S_{\omega}(\phi_{\omega}).
    \end{equation}
Next, let $v_{\omega} \in 
\mathcal{M}_{\omega}$. 
Then there exists a Lagrange
multiplier $\lambda _{\omega}\in \R$ 
such that
$S_{\omega}'(v_{\omega}) = \lambda_{\omega} 
K_{\omega}'(v_{\omega})$. 
Since $K_{\omega}(v_{\omega}) = 0$, 
we obtain 
$\langle S_{\omega}'(v_{\omega}), 
v_{\omega} \rangle = K_{\omega}(v_{\omega}) 
= 0$ and 
 $\langle K_{\omega}'(v_{\omega}), 
v_{\omega} \rangle = -(p-1) 
\|v_{\omega}\|_{L^{p+1}}^{p+1} \neq 0$. 
Hence, we find that $\lambda_{\omega} = 0$, 
that is, $S_{\omega}'(v_{\omega}) = 0$. 
Thus, by \eqref{eq1-3}, we obtain 
$S_{\omega}(\phi_{\omega}) 
\leq S_{\omega}(v_{\omega}) = d(\omega)$. 
This together with \eqref{eq2-9} yields that 
$S_{\omega}(\phi_{\omega}) 
= S_{\omega}(v_{\omega}) 
= d(\omega)$. 
This implies that $\phi_{\omega} \in 
\mathcal{M}_{\omega}$ and $v_{\omega} 
\in \mathcal{G}_{\omega}$. 
Therefore, we conclude that 
$\mathcal{M}_{\omega} 
= \mathcal{G}_{\omega}$. 
\end{proof}
Now, we give the proof of Theorem 
\ref{thm:1.1} \textrm{(i)}. 
\begin{proof}[Proof of Theorem 
\ref{thm:1.1} \textrm{(i)}] 
We see from the proof of Theorem 
\ref{thm-ex} that there exists a minimizer 
$\phi_{\omega} \in X$ 
of $d(\omega)$, which is 
positive, even and decreasing 
in $r = |x|$. 
By Theorem \ref{thm:2.2}, we see that 
$\phi_{\omega} \in \mathcal{G}_{\omega}$. 
This completes the proof. 
\end{proof}

\section{Positivity and evenness of ground state} 
\label{sec-po-gs}
In this section, we show that the
ground states are positive up to phase and
even functions. 

\begin{lemma}\label{prop-reg}
Assume that (V1). Then $\mathcal{A}_\omega\subset C^{1}(\R) 
\cap C^{3}(\R \setminus \{0\})$.   
\end{lemma}

\begin{proof}
The assertion follows from \cite[Theorem~11.7 (i) and (vi)]{LiebLoss2001}. 
\end{proof}

\begin{theorem}\label{thm:3.1}
Assume that (V1). 
Let $u\in\mathcal{M}_\omega$. There exists $\theta\in\R$ such that $e^{i\theta}u$ 
is a positive function.
\end{theorem}

\begin{proof}
Take $\theta\in\R$ such that $e^{i\theta}u(0)\ge 0$, and put $v:=e^{i\theta}u$.

First, we show the positivity of $v$. We put $v_1:=\lvert \Re v \rvert$ and $v_2:=\lvert \Im v \rvert$. Then we have $|v_1+iv_2|=|u|$ and $|\nabla(v_1+iv_2)|=|\nabla u|$, and so $v_1+iv_2\in \mathcal{M}_\omega$. Since $\mathcal{M}_\omega\subset\mathcal{A}_\omega$, $v_1+iv_2$ solves \eqref{sp1}. In, particular, $v_j$ solves 
\[  -v_j''+(\omega+V+|u|^{p-1})v_j=0,\quad 
    j=1, 2. \]
Moreover, Lemma~\ref{prop-reg} implies $v_j\in C^1(\R)\cap C^3(\R\setminus\{0\})$. Since $v_2(0)=v_2'(0)=0$, where $v_2'(0)=0$ follows from $v_2\ge 0$ and $v_2\in C^1(\R)$, the uniqueness of the solution of the initial value problem for the ODE implies $v_2\equiv 0$. Moreover, if $v_1(x_0)=0$ for some $x_0\in\R$, since $v_1\ge 0$ and $v_1\in C^1(\R)$, we have $v_1'(x_0)=0$, which implies $v_1\equiv 0$. This contradicts to $u\not\equiv0$. Thus $v_1>0$ on $\R$. Moreover, since $|v|=v_1>0$, $v(0)>0$, and $v$ is continuous, we see that $v>0$ on $\R$. 
\end{proof}

\begin{lemma} \label{lem:u=u^*}
Assume (V1). 
If a positive function $u\in X$ satisfies 
\begin{align*}
   &(u^*)'(r)<0\quad\text{for all }r>0, &
   &\|u_x\|_{L^2}=\|(u^*)_x\|_{L^2},&
   &\int_{\R} V|u|^2 dx =\int_{\R^d} V|u^*|^2 dx,
\end{align*}
where $u^{*}$ is the Schwarz rearrangement of 
the function $u$, then $u=u^*$. 
\end{lemma}

\begin{proof}
Since $(u^*)'<0$, we have 
$\{x\in\R\colon 0<u^*(x)<\|u^*\|_{L^\infty},\: |\nabla u^*(x)|=0\}=\emptyset$.
Therefore, by applying \cite[Theorem~1.1]{BurcFero15} together with 
$\|u_x\|_{L^2}=\|(u^*)_x\|_{L^2}$, we obtain $u=u^*(\cdot-y)$ for some $y\in\R$.

Suppose $y\ne 0$ and derive a contradiction. By the strict monotonicity of $u^*$, we see that 
$
\{x\in\R\colon u^*(x)^2>s\}=B_{R(s)}(0)
$
for some continuous and strictly decreasing function $R(s)$ with $R(s)\to\infty$ as $s\downarrow 0$, where $B_R(y)$ denotes the open ball centered at $y$ with radius $R$.  
By the monotonicity of $V$, we have 
$\int_{B_{R(s)}(0)}V(x)\,dx \le \int_{B_{R(s)}(y)}V(x)\,dx$ for all $s>0$. Moreover, since $V'(r_0)>0$ for some $r_0>0$, we can take $s_0>0$ sufficiently small so that $r_0\in B_{R(s_0)}(0)$. Then it is clear that
$
\int_{B_{R(s_0)}(0)}V(x)\,dx < \int_{B_{R(s_0)}(y)}V(x)\,dx.
$
Thus, noting that $s\mapsto \int_{B_{R(s)}(0)}V(x)\,dx$ and $s\mapsto\int_{B_{R(s)}(y)}V(x)\,dx$ are continuous because $V\in L_{\mathrm{loc}}^1(\R)$, we obtain
\begin{align*}
    \int_{\R} V(x) |u^*|^2 dx
   &= \int_0^\infty \int_{B_{R(s)}(0)}V(x)\,dx\,ds
    < \int_0^\infty \int_{B_{R(s)}(y)}V(x)\,dx\,ds
\\ &= \int_{\R} V(x) u^*(x-y)^2\,dx
    = \int_{\R} V(x)|u|^2\,dx,
\end{align*}
which contradicts $\int_{\R} V(x)|u^*|^2\,dx=\int_{\R} V(x)|u|^2\,dx$. Therefore, $y=0$.  
This completes the proof.
\end{proof}

\begin{theorem}\label{prop-ev}
Assume that (V1). 
Let $u\in\mathcal{M}_\omega$ be a positive ground state. Then $u$ is even and strictly decreasing. 
\end{theorem}

\begin{proof}
First, by $\|(u^*)_x\|_{L^2}\le \|u_x\|_{L^2}$ and \eqref{eq:Vreaineq}, we have $Q_\omega(u^*) \le Q_\omega(u)=d(\omega)=N(u)=N(u^*)$, which, by Lemma~\ref{lem:2.2}, implies $u^*\in\mathcal{M}_\omega$. In particular, $u^*\in\mathcal{A}_\omega$, and moreover $\|u_x\|_{L^2}=\|(u^*)_x\|_{L^2}$ and 
$\int_{\R} V|u|^2 dx =\int_{\R} V|u^*|^2 dx$.

Next, we show that $(u^*)'<0$ on $(0,\infty)$. Clearly $(u^*)'\le 0$. Suppose $(u^*)'(r_0)=0$ for some $r_0>0$. Then $(u^*)''(r_0)\ge 0$ and $u^*(r_0)>0$. We put $e(r):=\tfrac12 (u^*)'(r)^2-\tfrac12(\omega+V(r))u^*(r)^2+\tfrac1{p+1}u^*(r)^{p+1}$. Since $u\in X$, we have $\int_0^\infty e(r)\,dr<\infty$. Moreover, by \eqref{sp1}, $e'(r)=-\tfrac12 V'(r)(u^*)^2\le 0$. Thus, $\lim_{r\to\infty}e(r)=0$ and $e(r)\ge 0$ for all $r\ge 0$. In particular, $\tfrac12(\omega+V(r_0))u^*(r_0)^2\le \tfrac1{p+1}u^*(r_0)^{p+1}$. Therefore, 
\[
(u^*)''(r_0)=(\omega+V(r_0))u^*(r_0)-u^*(r_0)^{p}\le -\frac{p-1}{p+1}u^*(r_0)^{p}<0,
\]
which contradicts $(u^*)''(r_0)\ge 0$. Hence $(u^*)'<0$ on $(0,\infty)$.

Therefore, by Lemma~\ref{lem:u=u^*}, we obtain $u=u^*$ and $u'(r)<0$ on $(0,\infty)$. This completes the proof.
\end{proof}

\section{Uniqueness of ground state} 
\label{sec-unique}
This section is devoted to the following: 
\begin{theorem}\label{thm-unique}
 Assume that (V1). 
 Then the positive solution $\phi_{\omega} 
\in X$ to \eqref{sp1} is unique. 
\end{theorem}

Let $\phi_{\omega} \in X$ 
be a positive solution 
to \eqref{sp1}. 
Then we see from Theorem  
\ref{prop-ev} that $\phi_{\omega}$ 
becomes even 
and $\phi_{\omega}'(x) < 0$ for all $x > 0$. 
Thus, it suffices to consider 
even positive solutions to \eqref{sp1}.   
Let $r = |x|$. 
This together with Lemma~\ref{prop-reg}
implies that 
\eqref{sp1} is reduced to the following: 
\begin{equation} \label{rsp1}
    \left\{\begin{alignedat}{2}
   &\phi'' - (V(r) + \omega) \phi + \phi^{p} = 0 &\quad 
   &\mbox{on }(0, \infty),  
\\ & \phi'(0) = 0 & 
    \end{alignedat}\right.
\end{equation}
where the prime mark denotes the differentiation 
with respect to $r$. 
We will adapt the argument of 
\cite[Theorem 3.3]{MR4191499}
who proved the uniqueness of positive 
solution to semilinear elliptic equations 
in one dimension by using the argument of \cite{ShioWata13, ShioWata16}. 
To this end, 
we consider the general ordinary differential equation: 
\begin{equation}\label{gsp}
    \phi''+ 
    \frac{f'(r)}{f(r)} \phi'
    - g(r) \phi + h(r) \phi^{p} 
    = 0. 
\end{equation}
We put 
\[  \begin{split}
   & a(r) := f(r)^{2(p+1)/(p+3)} h(r)^{-2/(p+3)}, \\
   & b(r) := - \frac{1}{2} a^{\prime}(r) + \frac{f^{\prime}(r)}{f(r)} 
   a(r), \\
   & c(r) := - b^{\prime}(r) + \frac{f^{\prime}(r)}{f(r)} 
   b(r), \\
   & G(r) := b(r) g(r) + \frac{1}{2} c^{\prime}(r) 
   -  \frac{1}{2} (a g)^{\prime}(r). 
    \end{split} \]

Let $\phi_{\omega}$ be a positive solution 
to \eqref{gsp}. 
We put 
    \[
    \begin{split}
    J(r;\phi) 
    := 
    \frac{1}{2} a(r) (\phi'(r))^2 
    + b(r) \phi'(r) \phi(r) 
    + \frac{1}{2} c(r) \phi^2(r) 
    - \frac{1}{2} a(r) g(r) \phi^2(r) 
    + \frac{1}{p + 1} a(r) h(r) 
    \phi^{p + 1}(r). 
    \end{split}
    \]
Then it is known that 
the following Pohozaev identity holds: 

\begin{proposition}\label{prop-poho}
Let $\phi$ be a positive solution to \eqref{gsp}. Then we have 
\begin{equation}
    \frac{d}{d r} J(r; \phi) 
    = G(r) \phi^{2}(r)
\end{equation}
for all $r > 0$. 
\end{proposition}

We now modify the result of \cite[Theorem 3.3]{MR4191499}. 
To state the result, we prepare the following:  
\[  \begin{split}
  & U_{1}(r) := \frac{1}{f(r)} \int_{0}^{r} f(\tau) 
  (|g(\tau)| + h(\tau)) d \tau, \\
  & U_{2}(r) := \frac{1}{f(r)} \int_{0}^{r} f(\tau) h(\tau) d\tau, \\
  & D(r) := b(r)^{2} - a(r) (c(r) - a(r) g(r)). 
  \end{split}   \]
Then we impose the following conditions: 
\begin{enumerate}
\renewcommand{\labelenumi}{(\Roman{enumi})}
\item $f, g \in C^{3}(0, \infty)$ are positive functions and $g \in C^{1}(0, \infty)$. 
\item There exists $R > 0$ such that 
\begin{enumerate}
\renewcommand{\labelenumi}{(\alph{enumii})}
\item $f(|g| + h) \in L^{1}(0, R)$, 
\item $\tau \to f(\tau) (|g(\tau)| + h(\tau)) \int_{\tau}^{R} f^{-1}(\sigma) d \sigma \in L^{1}(0, R)$. 
\end{enumerate}
\item $\lim_{r \to 0} a(r) U_1(r) U_2(r) = \lim_{r \to 0} b(r)U_2(r) = 0$. 
\item $G\le 0$ and $G\not\equiv 0$. 
\end{enumerate}
Then the following theorem holds: 
\begin{theorem}\label{thm-unique-1}
Let $p > 1$ and 
assume that 
conditions (I)--(IV).
Then the problem \eqref{gsp} has at 
most one positive solution $\phi$ 
which satisfies $J(r;\phi) \to 0$ as $r \to \infty$ and $\phi'(0) = 0$. 
\end{theorem}

\begin{proof}
By Proposition~\ref{prop-poho} and (IV), for any positive solution $\phi$ to \eqref{gsp} satisfying $J(r;\phi)\to 0$ as $r\to\infty$, we have $J(r;\phi_{\omega}) \geq 0$ for all $r>0$. Hence, the assertion follows from the same argument as in the proof of \cite[Theorem~3.3]{MR4191499}.
\end{proof}

\begin{remark}
In our setting, \cite[Theorem~3.3]{MR4191499} cannot be applied directly, since the assumption \cite[(IV)]{MR4191499} may fail. These assumptions were used there to prove the positivity of $J(r;\phi)$. However, under our stronger assumption~(IV), the positivity of $J(r;\phi)$ follows immediately, so \cite[(IV)]{MR4191499} is not required.
\end{remark}

\begin{proof}[Proof of Theorem \ref{thm-unique}]
First, we can easily find that 
\eqref{rsp1} coincides \eqref{gsp} with 
  \begin{equation} \label{c1}
  f \equiv 1, \qquad 
  g(r) = V(r) + \omega, \qquad 
  h \equiv 1. 
  \end{equation}
Thus, we find that \eqref{rsp1} corresponds to 
  \begin{equation} \label{c2}
  a \equiv 1, \qquad \qquad 
  b \equiv 0, \qquad \qquad c \equiv 0 
  \end{equation}
and 
  \begin{equation} \label{c4}
  G(r) = - \frac{1}{2} V^{\prime}(r). 
  \end{equation}
Furthermore, we see that 
\eqref{rsp1} corresponds to 
  \begin{equation} \label{c3}
  \begin{split}
  & U_{1}(r) 
  = \frac{1}{f(r)} \int_{0}^{r} f(\tau) 
  (|g(\tau)| + h(\tau)) d \tau
  = \int_{0}^{r} (|V(\tau) + \omega| + 1) d\tau, \\
  & U_{2}(r) = 
  \frac{1}{f(r)} \int_{0}^{r} f(\tau) h(\tau) d\tau = r.
  \end{split}
    \end{equation}

It follows from (V1) and \eqref{c1} that $f, h \in C^{3}(0, \infty)$ are positive functions and $g \in C^{1}(0, \infty)$.   Thus, condition (I) holds. 

Next, we will check that condition (II) is satisfied. 
  It follows from (V1) and \eqref{c1} that 
   \[
   f(|g| + h) = (|V(r) + \omega| + 1) \in L^{1}(0, R). 
   \] 
In addition, by (V1) and \eqref{c1}, one 
has
    \[
    \tau \mapsto f(\tau) (|g(\tau)| + h(\tau)) 
    \int_{\tau}^{R} f(\sigma)^{-1} d\sigma
    = (|V(\tau) + \omega| + 1) (R - \tau) \in L^{1}(0, R). 
    \] 
Thus, we can verify condition (II). 

 By \eqref{c1}, \eqref{c2}, \eqref{c3}, and (V1), we obtain 
\[  \lim_{r \to 0} a(r) U_{1}(r) U_{2}(r) 
    = \lim_{r \to 0} r\int_{0}^{r} (|V(\tau) + \omega| + 1) d\tau 
    = 0, \qquad 
    \lim_{r \to 0} b(r) U_{2} (r) = \lim_{r \to 0} 0 = 0. \]
Therefore, we see that condition (III) is satisfied. 
 
By \eqref{c4} and (V1), we have (IV). 

Finally, we show that any positive solution $\phi\in X$ of \eqref{rsp1} satisfy $J(r; \phi)\to0$ as $r\to\infty$. we see that 
\begin{equation} \label{c5}
    J(r;\phi)
    = \frac{1}{2} \phi'(r)^{2} 
    - \frac{g (r)}{2} \phi(r) ^{2}
    + \frac{1}{p+1} \phi(r)^{p+1}. 
\end{equation}
Since $\phi\in X$, we have $J(r; \phi_\omega)\in L^1(0, \infty)$. This means that there exists $(r_n)_n$ such that $r_n\to\infty$ and $J(r_n; \phi) \to 0$ as $n\to \infty$. Since $J$ is decreasing from Proposition~\ref{prop-poho} and (IV) we obtain $J(r; \phi) \to 0$ as $r\to \infty$.
Therefore, the assertion follows from Theorem~\ref{thm-unique-1}.
\end{proof}

We conclude this section with the proof of Theorem \ref{thm:1.1} \textrm{(ii)}. 
\begin{proof}[Proof of Theorem \ref{thm:1.1} \textrm{(ii)}]
Let $u \in \mathcal{G}_{\omega}$. 
Then by Theorems \ref{thm:2.2} and 
\ref{thm:3.1}, there exists 
$\theta \in \R$ such that 
$e^{i \theta} u$ is a positive solution to \eqref{sp1}. 
Since we see from Theorem \ref{thm-unique} that 
the positive solution to \eqref{sp1} 
is unique, we have 
$e^{i \theta} u = \phi_{\omega}$, 
where $\phi_{\omega}$ is the one given in 
Theorem \ref{thm:1.1} \textrm{(i)}. 
This completes the proof. 
\end{proof}

\section{Non-degeneracy of the ground state}
\label{sec-nond}

\subsection{Non-degenracy of the 
ground state in $X_{\text{rad}}$}
Throughout this subsection, we consider
only the real-valued function. 
\begin{theorem}\label{thm-nond}
Assume that (V1). Let $\phi_{\omega} 
\in X$ be the positive 
ground state of \eqref{sp1}. 
Then $\phi_{\omega} \in X$ 
is non-degenerate in $X_{\text{rad}}$. 
\end{theorem}
\begin{proof}[Proof of Theorem \ref{thm-nond}]
We still employ the notations \eqref{c1}, \eqref{c2}, and \eqref{c4} in Section \ref{sec-unique}. 
It follows from condition (V1) that 
$G = - \frac{1}{2} V^{\prime} \not\equiv 0$ in $(0, \infty)$.  
This implies that 
we can find a closed interval $[r_{1}, r_{2}] \subset (0, \infty)$ such that 
    \[
    \min_{r \in [r_{1}, r_{2}]} |G(r)| > 0. 
    \]
We choose $\gamma \in C_{0}^{\infty} (0, \infty) \setminus \{0\}$ such that 
$\gamma \geq 0$ and $\operatorname{supp}\gamma = [r_{1}, r_{2}]$. 
We take $\delta > 0$ and fix it later. 
We define 
    \[
    \begin{split}
    & 
    g_{\delta}(r) := g(r) + \delta \gamma(r) h(r) \phi_{\omega}^{p-1}(r) 
    = V(r) + \omega + \delta \gamma(r) \times 1 \times \phi_{\omega}^{p-1}(r) 
    = V(r) + \omega + \delta \gamma(r) \phi_{\omega}^{p-1}(r), \\ 
    & 
    h_{\delta}(r) 
    := (1 + \delta \gamma(r)) h(r) 
    =1 + \delta \gamma(r)
    \end{split}
    \]
in $(0, \infty)$. 
We define $a_{\delta}, b_{\delta}, c_{\delta}, G_{\delta}$ and $D_{\delta}$ 
in $(0, \infty)$ as follows: 
    \[
    \begin{split}
    & a_{\delta} := f^{\frac{2(p+1)}{p + 3}} h_{\delta}^{- \frac{2}{p+3}} 
    = h_{\delta}^{- \frac{2}{p+3}}, 
    \qquad
    b_{\delta} := - \frac{1}{2} a_{\delta}' 
    + \frac{f'}{f} a_{\delta} 
    = - \frac{1}{2} a_{\delta}', \\
    & 
    c_{\delta} := - b_{\delta}' + \frac{f'}{f} b_{\delta} 
    = - b_{\delta}', 
    \qquad 
    G_{\delta} := b_{\delta} g_{\delta} + \frac{1}{2} c_{\delta, r} 
    - \frac{1}{2} (a_{\delta} g_{\delta})', 
    \\
    & 
    D_{\delta} := b_{\delta}^{2} - a_{\delta} (c_{\delta} - a_{\delta} g_{\delta}). 
    \end{split}
    \]
Since $g_{\delta} = g = V + \omega$ and $h_{\delta} = h = 1$ in 
$(0, \infty) \setminus [r_{1}, r_{2}]$, we can easily see that 
    \[
    a_{\delta} = a = 1, \qquad 
    b_{\delta} = b = 0, \qquad 
    c_{\delta} = c = 0, \qquad 
    G_{\delta} = - \frac{1}{2} a g = G, 
    \qquad 
    D_{\delta} = 
    g = D 
    \]
in $(0, \infty) \setminus [r_{1}, r_{2}]$. 
We fix $\delta > 0$ small enough such that 
    \[
    \min_{r \in [r_{1}, r_{2}]} |G_{\delta}(r)| > 0. 
    \]
We see that 
    \begin{equation} \label{eq-nond5}
    G_{\delta}(r) \leq 0 \qquad 
    \mbox{in $(0, \infty)$}. 
    \end{equation}
Indeed, if $r \not\in [r_{1}, r_{2}]$, 
we have by condition (V1) that 
$G_{\delta} (r) = G(r) = - \frac{1}{2} V^{\prime}(r) \leq 0$. 
In addition, since $\min_{r \in [r_{1}, r_{2}]} |G_{\delta}(r)| > 0$ and $G(r) \leq 0$ 
in $(0, \infty)$, we find that \eqref{eq-nond5} holds. 
Note that $\phi_{\omega}$ satisfies 
    \[
    \begin{split}
    \phi_{\omega}^{\prime \prime} - g_{\delta}(r) \phi_{\omega} 
    + h_{\delta}(r) \phi_{\omega}^{p}
    & 
    = \phi_{\omega}^{\prime \prime} - (V(r) + \omega) \phi_{\omega} 
    - \delta \gamma(r) \phi_{\omega}^{p} 
    +  \phi_{\omega}^{p} + \delta \gamma(r) \phi_{\omega}^{p} \\
    & 
    = \phi_{\omega}^{\prime \prime} - (V(r) + \omega) \phi_{\omega} 
    +  \phi_{\omega}^{p} = 0. 
   \end{split}
    \]
 $\phi_{\omega}(0) \in (0, \infty)$ and $\lim_{r \to \infty} \phi_{\omega}(r) = 0$. 
Namely, $\phi_{\omega}$ is the positive solution to 
    \begin{equation}\label{eq-nod6}
    \begin{cases}
     \phi_{\omega}^{\prime \prime} - g_{\delta}(r) \phi_{\omega} 
    + h_{\delta}(r) \phi_{\omega}^{p} = 0, \qquad 0< r < \infty, \\
    \lim_{r \to \infty} \phi_{\omega}(r) = 0. 
    \end{cases}
    \end{equation}
We can find that positive solution to \eqref{eq-nod6} is unique for sufficiently small 
$\delta > 0$ from 
Section \ref{sec-unique} and \eqref{eq-nond5}. 

We shall show that $\phi_{\omega}$ is non-degenerate in $X_{\text{rad}}$. 
Suppose to the contrary that $\phi_{\omega}$ is degenerate in $X_{\text{rad}}$. 
Then there exists $\psi_{\omega} \in X_{\text{rad}} \setminus \{0\}$ such that 
    \begin{equation} \label{eq-nod7}
    \psi_{\omega}^{\prime \prime} - g(r) 
    \psi_{\omega} + p \phi_{\omega}^{p-1} 
    \psi_{\omega} = 0 
    \qquad \mbox{for $0 < r < \infty$}. 
    \end{equation}
We define $C^{2}$-functional $S_{\omega, \delta}$ by 
\[  S_{\omega, \delta}(u) 
    = 2\int_{0}^{\infty}
    \Bigl(\frac{1}{2} (|u^{\prime}|^{2} + g_{\delta}(r) |u|^{2}) 
    - \frac{1}{p + 1} h_{\delta} (r) |u|^{p + 1}\Bigr) dr 
    \quad 
    \mbox{for }u \in X. \]
We can easily see that 
\[  \langle S_{\omega, \delta}''(\phi_{\omega})v, v\rangle
    =\langle S_{\omega}''(\phi_{\omega})v, v\rangle  
    - 2 \delta (p-1) \int_{0}^{\infty} \gamma(r) \phi_{\omega}^{p-1}(r) v^{2}(r) dr \]
 for each $v \in X_{\text{rad}}$ and 
\[  \langle S_{\omega}''(\phi_{\omega})\phi_{\omega}, \phi_{\omega}\rangle 
    = - 2 (p-1) \int_{0}^{\infty}\phi_{\omega}^{p+1}(r) dr < 0. \]
Moreover, we have 
\begin{equation} \label{eq-nod8}
    \int_{0}^{\infty} \gamma(r) \phi_{\omega}^{p-1} (r) \psi_{\omega}^{2}(r) dr > 0. 
\end{equation}
Indeed, if not, we have
$\int_{0}^{\infty} \gamma(r) \phi_{\omega}^{p-1} (r) \psi_{\omega}^{2}(r) dr = 0$. 
Since $\phi_{\omega}(r) > 0$ for any $r> 0$, 
we see that $\psi_{\omega} \equiv 0$ on $\operatorname{supp}\gamma$. 
From the uniqueness of the initial value problem of \eqref{eq-nod7}, 
we see that $\psi_{\omega} \equiv 0$ in $(0, \infty)$, which contradicts 
$\psi_{\omega} \in X_{\text{rad}} \setminus \{0\}$. 
Thus, by \eqref{eq-nod8}, we obtain 
\begin{equation} \label{eq-nod9}
    \begin{split}
    \langle S_{\omega, \delta}''(\phi_{\omega})\alpha \phi_{\omega} + \beta \psi_{\omega}, \alpha \phi_{\omega} + \beta \psi_{\omega}\rangle 
   ={}&\mathopen{} - 2 (p-1) \int_{0}^{\infty}\phi_{\omega}^{p+1}(r) dr \alpha^{2} 
\\ &- 2 \delta (p-1) \int_{0}^{\infty} \gamma(r) \phi_{\omega}^{p+1}(r) dr \alpha^{2} 
\\ &- 2 \delta (p-1) \int_{0}^{\infty} 
    \gamma(r)
    \phi_{\omega}^{p-1}(r) 
    \psi_{\omega}^{2}(r) dr \beta^{2} 
\\ & - 4 \delta (p-1) 
    \int_{0}^{\infty} \gamma(r) \phi_{\omega}^{p-1}(r) \phi_{\omega}(r) \psi_{\omega}(r) dr
    \alpha \beta
    < 0
    \end{split}
\end{equation}
for each $(\alpha, \beta) \in \R^{2} \setminus \{(0, 0)\}$. 
In addition, we can easily verify that 
$\phi_{\omega}$ does not satisfy 
\eqref{eq-nod7}, so that 
$\phi_{\omega}$ and 
$\psi_{\omega}$ 
are linearly independent. 
Thus, we see from \eqref{eq-nod9} that 
the Morse index of
$S_{\omega, \delta}^{\prime \prime
}(\phi_{\omega})$ 
is at least two.
However, since the Morse index of $S_{\omega, \delta}^{\prime \prime}(\phi_{\omega})$ is 
one (see Section \ref{sec-Morse} below), we obtain a contradiction. 
Hence, we have shown that $\phi_{\omega}$ is non-degenerate in $X_{\text{rad}}$. 
\end{proof}

\subsection{Proof of Theorem \ref{thm:1.1} 
\textrm{(iv)}}
Here, we consider not the real-valued 
functions but the complex-valued ones. 
We define the following operators 
$L_{\omega, +}$ and $L_{\omega, -}$ by 
\begin{align*}
    L_{\omega, +} 
   &= - \p_{xx} + V(x) 
    + \omega - p \phi_{\omega}^{p-1}, &
    L_{\omega, -} 
   &= - \p_{xx} + V(x) 
    + \omega - p \phi_{\omega}^{p-1}.
\end{align*}
For each $\varphi \in X$, we put 
$\varphi_{1} = \Re \varphi$ and 
$\varphi_{2} = \Im \varphi$. 
Then one has  
\[  S_{\omega}''(\phi_{\omega})\varphi 
    = L_{\omega, +} \varphi_{1} 
    + i L_{\omega, -} \varphi_{2}.  \]
In order to prove $\ker S_\omega''(\phi_{\omega})=\operatorname{span}
\{i\phi_\omega\}$, it suffices to show that 
$\ker L_{\omega, +} = \{0\}$ and 
$\ker L_{\omega, -} = \operatorname{span}
\{\phi_{\omega}\}$. 
To this end, we prepare the following lemma:
    \begin{lemma}\label{lem-sim}
       All eigenvalues of the 
       operators $L_{\omega, +}$ 
       and $L_{\omega, -}$ are simple. 
    \end{lemma}
\begin{proof}
Let $\varphi_{\omega} \in X$ 
and $\psi_{\omega} \in X$ 
be the eigenfunctions of $L_{\omega, +}$ 
associated with an eigenvalue $\lambda$. 
Then we have 
    \[
    \p_{x}( \varphi_{\omega} \p_{x} \psi_{\omega} 
    - \p_{x} \varphi_{\omega}
    \psi_{\omega}) 
    = \varphi_{\omega} \p_{xx} \psi_{\omega} 
    - \p_{xx} \varphi_{\omega} \psi_{\omega} 
    = 0. 
    \]
This implies that there exists a constant 
$C$ such that 
$\varphi_{\omega} \p_{x} \psi_{\omega} 
    - \p_{x} \varphi_{\omega}
    \psi_{\omega} = C$ on $\R$. 
Since $\lim_{|x| \to \infty} 
(\varphi_{\omega} \p_{x} \psi_{\omega} 
    - \p_{x} \varphi_{\omega}
    \psi_{\omega}) = 0$, 
    one has $C = 0$. 
Thus, $\varphi_{\omega}$ 
and $\psi_{\omega}$ are linearly dependent. 
We can prove the simpleness of the eigenvalues 
of $L_{\omega, -}$ similarly. 
Thus, we obtained the desired result. 
\end{proof}

\begin{proof}[Proof of Theorem \ref{thm:1.1} 
\textrm{(iv)}]

We first consider the operator $L_{\omega, +}$. 
Suppose that $\varphi_{\omega} \in X$ satisfies 
$L_{\omega, +} \varphi_{\omega} = 0$. 
Then we decompose $\varphi_{\omega}$ into 
the even part and the odd part as 
    \[
    \varphi_{\omega} = 
    \varphi_{1, \omega} + \varphi_{2, \omega}, 
    \qquad 
    \varphi_{1, \omega} = 
    \frac{\varphi_{\omega}(x) + 
    \varphi_{\omega}(-x)}{2}, \qquad 
    \varphi_{2, \omega} = 
    \frac{\varphi_{\omega}(x)  
    - \varphi_{\omega}(-x)}{2}. 
    \]
Then we see that $\varphi_{1, \omega} 
\in X_{\text{rad}}$ and also satisfies
$L_{\omega, +} \varphi_{1, \omega} = 0$. 
By Theorem \ref{thm-nond}, we have 
$\varphi_{1, \omega} = 0$. 
Observe that $\varphi_{2, \omega}$ satisfies 
$\varphi_{2, \omega}(0) = 0$ and 
    \[
    - \varphi_{2, \omega}'' + 
    V(r) \varphi_{2, \omega} + \omega 
    \varphi_{2, \omega} - p \phi_{\omega}^{p-1}
    (r) \varphi_{2, \omega} = 0, 
    \qquad r > 0.
    \]
Then by the similar argument as in \cite[Proof of Theorem 4.4]{MR4191499}, we obtain 
$\varphi_{2, \omega} = 0$. 
Thus, we see that $\ker L_{\omega, +} = \{0\}$.

Next, we study the operator $L_{\omega, -}$. 
We can easily find that 
$L_{\omega, -} \phi_{\omega} = 0$. 
Namely, $\phi_{\omega}$ is the zero eigenfunction 
of $L_{\omega, -}$. 
It follows from Lemma \ref{lem-sim} 
that $\ker L_{\omega, -} = \operatorname{span}
\{\phi_{\omega}\}$. 
This completes the proof. 
\end{proof}

\section{Proof of stability of the ground 
state}
\label{sec-sta}
\subsection{Spectral conditions}

\begin{proof}[Proof of Theorem \ref{thm:1.1} (iii) and (v)]
Let $\omega > \omega_{1}$.
We consider the following minimization problem: 
\[  \nu_\omega:= 
    \inf\{
    \langle L_{\omega, +} u, u \rangle 
    \colon u \in X,\: \|u\|_{L^{2}} = 1\}.  \]
Since $\langle L_{\omega, +} \phi_{\omega}, 
\phi_{\omega} \rangle = - (p-1) \|\phi_{\omega}\|_{L^{p+1}}^{p+1} <0$, 
we have $\nu_\omega < 0$. 
Then using a similar argument to \cite[Lemma 8.1.11]{MR2002047},
there exists a minimizer $\chi_{\omega} 
\in X$ that satisfies $L_{\omega, +} 
\chi_{\omega} = \nu_\omega \chi_{\omega}$. 
In addition, it follows from 
the min-max theorem (see e.g. 
Reed and Simon~\cite[Theorem XIII.1]{MR493421}), 
$\nu_\omega$ is the 
first eigenvalue of $L_{\omega, +}$. 
By Lemma \ref{lem-sim}, we see that 
$\nu_\omega$ is simple. 
In addition, since $L_{\omega, -} 
\phi_{\omega} = 0$ and $\phi_{\omega}$ is positive, 
we find that $L_{\omega, -}$ is 
a non-negative operator. 
Thus, Theorem \ref{thm:1.1} \textrm{(iii)} 
holds. 

Next, we will show Theorem \ref{thm:1.1} (v).
Since $\omega > \omega_{1}$, 
we see that the operator $-\p_{xx} + V(x) 
    + \omega$ is invertible. 
In addition, $\lim_{|x| \to \infty} \phi_{\omega}^{p-1}(x) 
= 0$.     
Thus, we see from the Weyl's essential 
spectrum that 
    $\sigma_{\text{ess}}(L_{\omega, \pm}) 
    = \sigma_{\text{ess}}(- \p_{xx}+ V(x) + \omega)
    = [\delta, \infty)$. 
This completes the proof
\end{proof}

\subsection{Slope condition}
In this section, we study 
the sign of 
$\frac{d}{d \omega} \|\phi_{\omega}\|_{L^2}^2$ 
and give the proof of 
Theorem \ref{thm-stability}.

In what follows, 
we investigate the sign of 
$\frac{d}{d \omega} \|\phi_{\omega}\|_{L^2}^2$.
We note that our argument extends to higher-dimensional cases, except for Lemma \ref{lem:v0posi} below.
Therefore, in order to clarify why condition 
$d = 1$ is necessary, 
we consider the problem in general dimensions.

First, we pay our attention to the 
regularity of $\phi_{\omega}$ with respect to 
$\omega > \omega_{1}$. 
Concerning this, we obtain the following: 
    \begin{lemma}\label{lem-reg}
    Let $d = 1$ and $1 < p < 2^{*} -1$. 
    Then the mapping $\omega \in (\omega_{1}, \infty) \mapsto \phi_{\omega} \in 
    X_{\text{rad}}$ is $C^{1}$.  
    \end{lemma}
We obtain Lemma \ref{lem-reg} by using 
a similar argument of Shatah and 
Strauss~\cite{MR804458}. 
By adapting the abstract 
result of Grillakis, Shatah and 
Strauss~\cite{MR901236} to 
our equation \eqref{NLS1}, 
we obtained the following: 
\begin{theorem} \label{s-sta}
Let $d = 1$ and $\phi_{\omega} \in \mathcal{G}_{\omega}$.
Assume that Assumption \ref{l-wellposed} and   
(V1). Then the 
standing waves $e^{i \omega t} 
\phi_{\omega}$ is stable if
$\frac{d}{d \omega} \|\phi_{\omega}\|_{L^2}^2 > 0$ and 
unstable if $
\frac{d}{d \omega} \|\phi_{\omega}\|_{L^2}^2
< 0$.  
\end{theorem}

We put
\begin{align*}
    u_\omega
   &:=\frac{\phi_{\omega}}{\|\phi_{\omega}\|_{L^2}},&
    v_\omega
   &:=\partial_\omega u_\omega. 
\end{align*}
Then we have 
\begin{align} \label{eq:iden}
   \|u_\omega\|_{L^2}
   &=1,&
    \int_{\R^d} u_\omega v_\omega dx
   &=\frac12\frac{d}{d\omega}\|u_\omega\|_{L^2}^2
    =0.
\end{align}
Moreover, $u_\omega$ satisfies the equation 
\begin{equation} \label{eq:sp2}
    -\Delta u_{\omega}
    +V(x) u_{\omega}
    +\omega u_{\omega}
    -\mu(\omega) u_\omega^p
    =0 \qquad \mbox{in $\R^{d}$}, 
\end{equation}
where $\mu(\omega):=\|\phi_{\omega}\|_{L^2}^{p-1}$. We note that the sign of $\mu'(\omega)$ coincides with that of $\frac{d}{d\omega}\|\phi_{\omega}\|_{L^2}^2$. 

\begin{lemma}\label{lem:Luvposi}
Let $d \geq 1$ and $1 < 
p < 2^{*} - 1$. 
For $\omega>\omega_1$, the following holds.
\begin{align}
    \label{eq:fm-mmp}
   &\mu(\omega)\int_{\R^d} 
   u_\omega^p v_\omega dx
    =\frac{1}{p-1}\bigl(1-\mu'(\omega)\|\phi_{\omega}\|_{L^{p+1}}^{p+1}\bigr),
\\  \label{eq:upvposi}
   &\int_{\R^d} u_\omega^p 
   v_\omega dx >0.
\end{align} 
\end{lemma}

\begin{proof}
Let $Q_\omega$ be a action functional of \eqref{eq:sp2} given by 
\[  I_\omega(v) 
    :=\frac{1}{2}\|\nabla v\|_{L^2}^2
    +\frac12\int_{\R^d} V(x)|v|^2+\frac{\omega}{2}\|v\|_{L^2}^2
    -\frac{\mu(\omega)}{p+1}\|v\|_{L^{p+1}}^{p+1}\]
We note that 
$I_\omega'(u_\omega)=0$ 
and 
    \[
    I_\omega''
(u_\omega)|_{X(\R^d; \R)}=
-\Delta+V(x)+\omega-p\mu(\omega)
u_\omega^{p-1} 
= -\Delta+V(x)+\omega-p \phi_{\omega}^{p-1}
= S_\omega''
(\phi_{\omega})|_{X(\R^d; \R)}. 
    \]
In particular, 
\begin{equation} \label{eq:Jomexp1}
    I_\omega''(u_\omega)
    u_\omega
    =-(p-1)\mu(\omega) 
    u_\omega^{p}.
\end{equation}
By differentiating the equation 
$I_\omega'(u_\omega)=0$ 
with respect to $\omega$, we obtain 
\begin{equation}\label{eq:Jomexp2}
    I_\omega''
    (u_{\omega}) v_\omega
    =-u_\omega + \mu'(\omega)u_{\omega}^{p}.
\end{equation} 
By using \eqref{eq:iden}, \eqref{eq:Jomexp1} 
and \eqref{eq:Jomexp2}, we have
\begin{align}
    \label{eq:exp-a}
   a&:=\langle 
   I_\omega''(u_\omega)
   u_\omega, 
   u_\omega \rangle 
    =-(p-1)\mu(\omega)\int_{\R^d} 
    \phi_{\omega}^{p+1} dx,
\\  \label{eq:exp-b}
   b&:=\langle I_\omega''
   (u_\omega)u_\omega, v_\omega\rangle 
    =\left\{\begin{aligned}
   &\mathopen{}-(p-1)\mu(\omega)\int_{\R^d} u_\omega^p v_\omega dx,
\\ &\mathopen{}\mu'(\omega)\int_{\R^d} u_\omega^{p+1} dx -1, 
    \end{aligned}\right.
\\  \label{eq:exp-c}
   c&:= \langle I_\omega''(
   u_\omega) v_\omega, 
   v_\omega\rangle 
    =\mu'(\omega)\int_{\R^d} 
    u_\omega^{p} v_\omega 
    dx. 
\end{align}
In particular, \eqref{eq:fm-mmp} holds. Let $\chi_\omega$ be the positive and normalized eigenfunction corresponds to the unique negative eigenvalue of $I_\omega''
(u_\omega)$. 
Since $u_\omega$ is also positive function, we have $\int_{\R} 
\chi_\omega u_\omega dx 
>0$. 
This means that there exists 
$t_{\omega} \in\R$ such that $\int_{\R} \chi_\omega
(t_{\omega} u_\omega + v_\omega) 
dx=0$. By using this,  
non-degeneracy of $I_\omega''
(u_\omega)|_{X_\mathrm{rad}(\R^d; \R)} = S_\omega''
(\phi_{\omega})|_{X_\mathrm{rad}(\R^d; \R)}$ and 
Theorem \ref{thm:1.1} \textrm{(iii)}--\textrm{(v)}, we have the following positivety:
\begin{align*}
  0&<\langle I_\omega''
  (u_\omega)(t_{\omega} 
  u_\omega + v_\omega), 
  t_{\omega} 
  u_\omega + v_\omega
  \rangle 
\\ &=at_{\omega}^2+2bt_{\omega}+c
    =a\Bigl(t_{\omega}+\frac{b}{a}\Bigr)^2
     -\frac{1}{a}(b^2-ac). 
\end{align*}
This together with $a<0$ yields that $0<b^2-ac$. We note from the expression \eqref{eq:exp-a}--\eqref{eq:exp-c} that $ac=b(b+1)$. Thus, 
$b = ac - b^2<0$. 
This means \eqref{eq:upvposi} holds.
\end{proof}

\begin{lemma}\label{lem:key1}
Let $d \geq 1$ and $1 < p < 
2^{*} - 1$.  
Then there exist positive constants $C_1, C_2>0$ such that for $\omega>\omega_1$,
\begin{equation} \label{eq:key1}
    C_1\mu'(\omega)\|u_\omega
    \|_{L^{p+1}}^{p+1}
    =C_2\Bigl(1+\frac{4}{d}-p\Bigr)
    -\int_{\R^d} (2V+x\cdot \nabla V)
    u_\omega v_\omega dx
\end{equation}
\end{lemma}

\begin{proof}
We put $\alpha:=d(p-1)/2$. By differentiating the relation
\begin{align*}
    I_\omega
    (t^{d/2} u_\omega(tx))
   &=\frac{t^{2}}{2}
   \|\nabla u_\omega\|_{L^2}^2
    +\frac{1}{2}\int_{\R^d} V(x/t)
    u_\omega^2 dx
    +\frac{\omega}
    {2}\|u_\omega\|_{L^2}^2
    -\frac{t^{\alpha}\mu(\omega)}{p+1}\|u_\omega\|_{L^{p+1}}^{p+1}
\end{align*}
at $t=1$ and using 
$I_\omega'(u_\omega)=0$, 
$\|\nabla u_\omega\|_{L^2}^2
=\mu(\omega)\|u_\omega\|_{L^{p+1}}^{p+1}-\int_{\R^d} Vu_\omega^2 dx -\omega\|u_\omega\|_{L^2}^2$ 
and $\|u_{\omega}\|_{L^{2}} = 1$, 
we have
\begin{align*}
  0&=\|\nabla u_\omega\|_{L^2}^2
    -\frac{1}{2}\int_{\R^d} 
    (x\cdot \nabla V) u_\omega^2 dx
    -\frac{\alpha\mu(\omega)}{p+1}|u|_{L^{p+1}}^{p+1}
\\ &=-\omega-\frac{1}{2}\int_{\R^d} 
(2V+x\cdot \nabla V)u_\omega^2 dx 
    +\Bigl(1-\frac{\alpha}{p+1}\Bigr)\mu(\omega)
    \|u_\omega\|_{L^{p+1}}^{p+1}. 
\end{align*}
Moreover, by differentiating 
the above identity with respect to $\omega$ and using \eqref{eq:fm-mmp}, we obtain
\begin{align*}
   &1+\int_{\R^d} 
   (2V+x\cdot \nabla V)
   u_\omega v_\omega dx
\\ &=\Bigl(1-\frac{\alpha}{p+1}\Bigr)\mu'(\omega)
\|u_\omega\|_{L^{p+1}}^{p+1}
    +\Bigl(1-\frac{\alpha}{p+1}\Bigr)(p+1)\mu(\omega)
    \int_{\R^d} u_\omega^{p} 
    v_\omega dx
\\ &=\Bigl(1-\frac{\alpha}{p+1}\Bigr)\mu'(\omega)
\|u_\omega\|_{L^{p+1}}^{p+1}
    +\Bigl(1-\frac{\alpha}{p+1}\Bigr)\frac{p+1}{p-1}\bigl(1-\mu'(\omega)
    \|u_\omega\|_{L^{p+1}}^{p+1}\bigr)
\\ &=-\Bigl(1-\frac{\alpha}{p+1}\Bigr)\frac{2}{p-1}
\mu'(\omega)
\|u_\omega\|_{L^{p+1}}^{p+1}
    +\Bigl(1-\frac{\alpha}{p+1}\Bigr)\frac{p+1}{p-1}.
\end{align*}
It follows from $\alpha = d(p-1)/2$ that   
\begin{align*}
   &\Bigl(1-\frac{\alpha}{p+1}\Bigr)\frac{2}{p-1}
    =\frac{d+2-(d-2)p}{(p-1)(p+1)}>0,
\\ &\Bigl(1-\frac{\alpha}{p+1}\Bigr)\frac{p+1}{p-1}-1
    =\frac{d}{2(p-1)}\Bigl(1+\frac{4}{d}-p\Bigr).
\end{align*}
Therefore, we see that 
\eqref{eq:key1} with 
$C_{1} = \frac{d+2-(d-2)p}{(p-1)(p+1)}$ and 
$C_{2} = \frac{d}{2(p-1)}$. 
\end{proof}

\begin{lemma} \label{lem:posione}
Let $d \geq 1$ and $1 < p < 2^{*} -1$. If $\mu'(\omega) \le 0$, then the set $[v_\omega > 0] := \{r \ge 0\colon v_\omega(r) > 0\}$ is a nonempty bounded interval. 
\end{lemma}

\begin{proof}
First, it is easy to see that $[v_\omega>0]\ne\emptyset$ due to $\int_{\R^d} u_\omega v_\omega dx =0$ (see \eqref{eq:iden}), $u_{\omega}(r) > 0$ 
for all $r > 0$, and $v_\omega\not\equiv 0$. 

Next, suppose that the set $[v_\omega>0]$ consists of more than one connected component.
Then there exist $0\le r_1<r_2\le r_3<r_4\le\infty$ such that 
\[  r_1=0 \text{ or } 
    v_\omega(r_1)=0, \quad
    v_\omega(r_2)
    = v_\omega(r_3)
    =v_\omega(r_4)=0,\quad 
    v_\omega>0 
    \text{ on } (r_1, r_2)\cup (r_3, r_4).  \]
We put $I_1:=[r_1, r_2)$, $I_2:=(r_3, r_4)$, and $v_j:= v_\omega 1_{I_j}$, where $1_A$ is the characteristic function of the set $A$. Then we see that $v_j\in X$ and there is $t\in \R$ such that $\int_{\R^d} \chi_\omega (v_1+t v_2) dx=0$, where $\chi_\omega$ is the positive and normalized eigenfunction corresponds to the unique negative eigenvalue of $I_\omega''(u_\omega)$. This implies 
\begin{equation} \label{eq:J''posi}
    \langle I_\omega''
    (u_\omega)(v_1 + 
    t v_2), v_1+t v_2\rangle\ge 0.
\end{equation}
On the other hand, by using the equation $I_\omega''(u_\omega)
v_\omega=-u_\omega+\mu'(\omega)
u_\omega^p$ and integration by parts, we obtain 
\begin{align*}
    \langle I_\omega''
    (u_\omega)
    v_j, v_j\rangle 
   &=\langle -
   u_\omega+\mu'(\omega)
   u_{\omega}^p, 
   v_j\rangle 
    <0,&
    \langle I_\omega''
    (u_\omega)
    v_1, v_2\rangle 
   &=0,
\end{align*}
where we used the assumption $\mu'(\omega)\le 0$, the positivity of 
$u_{\omega}, v_{j}$ 
and $\operatorname{supp} v_{1} \cap \operatorname{supp}v_{2} = \emptyset$.
Therefore, $\langle I_\omega''
    (u_\omega)(v_1 + 
    t v_2), v_1+t v_2\rangle<0$.
This contradicts \eqref{eq:J''posi}. 
Thus, the set $[v_\omega>0]$ is connected. 

Finally, suppose that $[v_\omega>0]$ is unbounded. Since $[v_\omega>0]$ is connected and $\int_{\R^d}u_\omega v_\omega=0$, it follows that $[v_\omega>0]=(r_1, \infty)$ for some $r_1>0$. Thus, by the positivity and the decreasing monotonicity of 
$u_\omega$, we have
\begin{align*}
    \int_0^{r_1}u_\omega^p v_\omega r^{d-1}dr
   &\le u_\omega(r_1)^{p-1}\int_0^{r_1} u_\omega 
    v_\omega r^{d-1}dr, &
    \int_{r_1}^\infty u_\omega^p v_\omega r^{d-1}dr
   &\le u_\omega(r_1)^{p-1}\int_{r_1}^\infty u_\omega 
    v_\omega r^{d-1}dr.
\end{align*}
Therefore, by $\int_{\R^d}u_\omega v_\omega=0$, we obtain
\begin{align*}
    \int_0^\infty u_\omega^p v_\omega r^{d-1} dr
    \le u_\omega(r_1)^{p-1} \int_0^\infty u_\omega v_\omega r^{d-1} dr
    =0.
\end{align*}
This contradicts \eqref{eq:upvposi}. Hence, $[v_\omega>0]$ is bounded.
\end{proof}

The following lemma is a key of our proof.

\begin{lemma} \label{lem:v0posi}
Let $1 < p < 2^{*} - 1$. 
If $d=1$ and $\mu'(\omega)\le 0$, then $v_\omega(0)>0$. 
\end{lemma}
\begin{proof}
Suppose that $v_\omega(0)\le 0$. By Lemma~\ref{lem:posione}, there exists $0\le r_1<r_2$ such that $v_\omega(r_1)=v_\omega(r_2)=0$ and $v_\omega>0$ on $(r_1, r_2)$. Since $u_\omega$ is a positive radial solution of \eqref{eq:sp2}, we have 
\begin{equation} \label{eq:sprad}
    u_\omega'' 
    +\frac{d-1}{r}u_\omega'
    -(V(r)+\omega) u_\omega 
    +\mu(\omega) u_\omega^p
    =0,\quad r>0.
\end{equation}
Differentiating \eqref{eq:sprad} with respect to $r$, we have 
\begin{equation} \label{eq:spu'}
    u_\omega'''
    +\frac{d-1}{r} u_\omega'' 
    -(V(r)+\omega) u_\omega'
    =\frac{d-1}{r^2} u_\omega'
    +V'(r) u_\omega
    -p\mu(\omega) 
    u_\omega^{p-1}
    u_\omega'.
\end{equation}
Differentiating \eqref{eq:sprad} with respect to $\omega$, we have 
\begin{equation} \label{eq:spv}
    v_\omega''
    +\frac{d-1}{r} v_\omega' 
    -(V(r)+\omega)v_\omega
    = u_\omega
    -p\mu(\omega) 
    u_\omega^{p-1}v_\omega
    -\mu'(\omega) u_\omega^p.
\end{equation}
By \eqref{eq:spu'} and \eqref{eq:spv}, we obtain 
\begin{align*}
   &\bigl((u_\omega''
   v_\omega - 
   u_\omega' 
   v_\omega')r^{d-1}\bigr)'
\\ &=\Bigl(\Bigl(
u_\omega'''+\frac{d-1}{r} u_\omega''\Bigr) 
v_\omega 
    -\Bigl(v_\omega'' + 
    \frac{d-1}{r} v_\omega'\Bigr)
    u_\omega' \Bigr)r^{d-1}
\\ &=\Bigl(\frac{d-1}{r^2} 
u_\omega' v_\omega
    +V'(r) u_\omega 
    v_\omega
    - u_\omega u_\omega'
    +\mu'(\omega)
    u_\omega^p 
    u_\omega'\Bigr)r^{d-1}. 
\end{align*}
Integrating the both side over $[r_1, r_2]$ and using $v_\omega(r_1) = v_\omega(r_2)=0$, we obtain 
\begin{align}\label{eq:u'v'intd} 
    \begin{aligned}
   &u_\omega'(r_1) v_\omega'(r_1)r_1^{d-1}
    - u_\omega'(r_2)
    v_\omega'(r_2)r_2^{d-1}
\\ &=\int_{r_1}^{r_2}\Bigl(\frac{d-1}{r^2}u_\omega' v_\omega
    +V'(r) u_\omega v_\omega
    - u_\omega u_\omega'
    +\mu'(\omega) u_\omega^p u_\omega'
    \Bigr)r^{d-1}\,dr.
    \end{aligned}
\end{align}
In particular, if $d=1$,
\begin{align} \label{eq:u'v'int1}
    \begin{aligned}
   &u_\omega'(r_1) v_\omega'(r_1)r_1^{d-1}
    - u_\omega'(r_2) v_\omega'(r_2)r_2^{d-1}
\\ &=\int_{r_1}^{r_2}\Bigl(
    V'(r) u_\omega v_\omega
    - u_\omega u_\omega'
    +\mu'(\omega) u_\omega^p u_\omega'
    \Bigr)r^{d-1}\,dr.
    \end{aligned}
\end{align}
Since $u_\omega'<0$, 
$v_\omega'(r_1)\ge 0$, and $v_\omega'(r_2)\le 0$, 
the left-hand side 
of \eqref{eq:u'v'int1} is non-positive. On the other hand, since 
$u_\omega>0$, 
$v_\omega>0$, 
$V'\ge 0$, $u_\omega'<0$ 
on $(r_1, r_2)$ and $\mu'(\omega)\le 0$, the right-hand side is positive. This is a contradiction. Thus, we obtain 
$v_\omega(0)>0$. 
\end{proof}

\begin{remark}
Our argument does not work in the case of higher dimensions $d\ge 2$. In this case, we  have the relation 
\eqref{eq:u'v'intd} instead of \eqref{eq:u'v'int1}. 
However, it is unclear whether 
the right-hand side of \eqref{eq:u'v'int1} is positive, since the term 
$\frac{d-1}{r^2}u_\omega' 
v_\omega$ is negative.
\end{remark}

\begin{proof}[Proof of Theorem \ref{thm-stability}]
From Theorem \ref{s-sta} and the fact that the sign of $\mu'(\omega)$ coincides with $\frac{d}{d \omega} \|\phi_{\omega}\|_{L^2}^2$, it suffces to show that $\mu'(\omega) > 0$ for all $\omega > \omega_1$. 
Suppose that there exists $\omega\in(\omega_1, \infty)$ such that $\mu'(\omega)\le 0$. Then \eqref{eq:key1} implies 
\begin{equation}\label{eq:posipotential}
    \int_{0}^\infty 
    (2V+rV')u_\omega v_\omega 
    dx
    \begin{cases}
    > 0 &  \text{if $1<p<5$},
\\  \ge 0 & \text{if $p=5$}.
    \end{cases}
\end{equation}
Now by Lemmas~\ref{lem:posione} and \ref{lem:v0posi}, we see that there exists $r_0>0$ such that $v_{\omega}>0$ 
on $[0, r_0)$ and 
$v_{\omega} \le 0$ on 
$(r_0, \infty)$. 
We see in fact that 
$v_{\omega}<0$ on 
$(r_0, \infty)$. Indeed, if 
$v_{\omega}(r_1)=0$ for 
some $r_1>r_0$, we see that 
$v_{\omega}'(r_1)=0$ and 
$v_{\omega}''(r_1)\le 0$, 
which contradict \eqref{eq:spv}. Therefore, by the assumption (V2) and \eqref{eq:iden}, 
we obtain 
\begin{align*}
    \int_{0}^\infty (2V+rV') u_\omega v_\omega dx
   &=\Bigl(\int_{0}^{r_0}+\int_{r_0}^\infty\Bigr) (2V+rV')u_\omega v_\omega dx
\\ &\begin{cases}
    \le (2V+rV')(r_0)\int_0^\infty u_\omega 
    v_\omega dx & \text{if $1<p<5$},
\\  <(2V+rV')(r_0)\int_0^\infty u_\omega 
v_\omega & \text{if $p=5$}
\end{cases}
\\ &=0,
\end{align*}
which contradicts \eqref{eq:posipotential}. Thus, $\mu'(\omega)>0$ for all $\omega>\omega_1$. This completes the proof.
\end{proof}

\appendix

\section{Bounded from below}
\label{sec-bb}
\begin{lemma} \label{lema-1}
Let $d=1$ and assume that 
\emph{(V1)}. 
Then the following assertions hold:
\begin{enumerate}
\renewcommand{\labelenumi}{(\roman{enumi})}
\item $\omega_1$ given by \eqref{eq1-2} is well-defined, namely,
\[  \inf\Bigl\{
    \int_{\R} \bigl(|u_{x}|^{2} + V|u|^{2}\bigr)\,dx
    \colon u \in X,\: \|u\|_{L^{2}} = 1\Bigr\} 
    > -\infty.  \]
\item 
For any $\omega>\omega_1$, the embedding \eqref{eq1-1} holds. 
\end{enumerate}
\end{lemma}

\begin{proof}
Let $u\in X$. By the Sobolev inequality, we have
\begin{equation} \label{eqA-1}
    \int_{|x| \le \delta} V|u|^{2} \, dx 
    \le \|u\|_{L^{\infty}}^{2} 
    \int_{|x| \le \delta} |V(x)| \, dx
    \le C \|u\|_{H^{1}}^{2} 
    \int_{|x| \le \delta} |V(x)| \, dx,  
\end{equation}
where the constant $C > 0$ does 
not depend on $u$. Since $V \in L_{\text{loc}}^{1}(\R)$, 
there exists $\delta > 0$ such 
that 
\[
   C \int_{|x| \le \delta} |V(x)| \, dx < \frac12.
\] 
Thus, by \eqref{eqA-1}, we obtain 
\begin{equation} \label{eqA-2}
    \begin{split}
    \int_{\R} V|u|^{2} \, dx 
    &= \int_{|x| \le \delta} V|u|^{2} \, dx 
      + \int_{|x| \geq \delta} V|u|^{2} \, dx 
\\  &\ge - \frac12 \|u\|_{H^{1}}^2
      + V(\delta)\int_{|x| \ge \delta} |u|^{2} \, dx 
\\  &\ge - \frac12 \|u_x\|_{L^2}^2
      - C \|u\|_{L^2}^2.  
    \end{split}
\end{equation}
In particular, (i) holds.

Moreover, by the definition of $\omega_1$, for $\omega>\omega_1$, we have 
\begin{equation}\label{eq:A-3}
    (\omega-\omega_1)\|u\|_{L^2}^2
    \le \|u\|_{X_{\omega}}^2.
\end{equation}
By \eqref{eqA-2},
\begin{align*}
    \|u_x\|_{L^2}^2
   &= \|u\|_{X_{\omega}}^2
    - \int_{\R} V|u|^2 \, dx
    - \omega \|u\|_{L^2}^2
\\ &\le \|u\|_{X_{\omega}}^2
    + \tfrac12 \|u_x\|_{L^2}^2
    + C \|u\|_{L^2}^2.
\end{align*}
This and \eqref{eq:A-3} imply
\[  \|u_x\|_{L^2}^2
    \le C\bigl(\|u\|_{X_{\omega}}^2
    + \|u\|_{L^2}^2\bigr)
    \le C \|u\|_{X_{\omega}}^2. \]
Therefore, combining this with \eqref{eq:A-3}, we obtain $\|u\|_{H^1}\le C\|u\|_{X_{\omega}}$. This means (ii). 
\end{proof}

\section{Morse Index of $S_{\omega, \delta}^{\prime \prime}(\phi_{\omega})$}
\label{sec-Morse}
In this section, we show that the Morse index of 
$S_{\omega, \delta}''(\phi_{\omega})$ equals to one. 
First, we consider the following minimization problem: 
	\[
	d(\omega, \delta):= \inf\left\{ 
	S_{\omega, \delta}(u) \colon u \in X \setminus \{0\}, \: 
	K_{\omega, \delta}(u) = 0
	\right\},  
	\]	
where 
    \[
	K_{\omega, \delta}(u) := \langle S_{\omega, \delta}^{\prime}(u), u \rangle 
	= 2\int_{0}^{\infty}
    \left((|u^{\prime}|^{2} + g_{\delta}(r) |u|^{2}) 
    - h_{\delta} (r) |u|^{p + 1}\right) dr 
    \qquad 
    \mbox{for $u \in X$}.
	\]
By a similar argument of Section \ref{sec-ex}, 
we can find that there exists a minimizer, which is 
positive and even solution to \eqref{eq-nod6}. 
We can easily verify that the positive 
function $\phi_{\omega}$ satisfies \eqref{eq-nod6}. 
Since positive solution to \eqref{eq-nod6} is unique 
(see Section \ref{sec-nond}), we see that 
$\phi_{\omega}$ is also a minimizer of $d(\omega, \delta)$. 
Using this, we shall show the following: 
\begin{lemma}\label{lem-morse1}
If $u \in H^{1}_{\text{rad}}(\R)$ satisfies 
$\langle K_{\omega, \delta}(\phi_{\omega}), u \rangle = 0$, we have $\langle S_{\omega, \delta}^{\prime \prime} 
(\phi_{\omega}) u, u \rangle \geq 0$. 
\end{lemma}
\begin{proof}
We shall show by contradiction. 
Suppose to the contrary that 
there exists $u_{*} \in H^{1}_{\text{rad}}(\R)$ such that 
$\langle K_{\omega, \delta}(\phi_{\omega}), u_{*} \rangle = 0$ and 
$\langle S_{\omega, \delta}^{\prime \prime} 
(\phi_{\omega}) u_{*}, u_{*} \rangle = -1$.
Define the function
$Z: \R^{2} \to H^{1}_{\text{rad}}(\R)$ by 
    \[
    Z(a, b) = \phi_{\omega} + a u_{*} + b \phi_{\omega}. 
    \]
We can easily verify that 
$K_{\omega, \delta}(Z(0, 0)) = K_{\omega, \delta}(\phi_{\omega}) = 0$. 
In addition, we have 
    \[
    \frac{\p}{\p b} K_{\omega, \delta}(Z(a, b))\biggl|_{(a, b) = (0, 0)}
    = \langle S_{\omega, \delta}^{\prime \prime}(\phi_{\omega}) \phi_{\omega}, 
    \phi_{\omega} \rangle 
    = - 2(p-1) \int_{0}^{\infty} h_{\delta}(r) |\phi_{\omega}|^{p+1} dx < 0. 
    \]
Hence, it follows from the implicit function theorem that 
there exists $a_{0} > 0$ and $h \in C^{2}(-a_{0}, a_{0})$ such that 
$K_{\omega, \delta}(Z(a, h(a))) = 0$ and $h(0) = 0$. 
Since $\phi_{\omega}$ is a minimizer of $d(\omega, \delta)$, we have 
$\frac{d^{2} S_{\omega, \delta}}{d a^{2}}(Z(a, b))\Bigl|_{a = 0} \geq 0$.     

On the other hand, 
differentiating $K_{\omega, \delta}(Z(a, h(a))) = 0$ and substituting $a = 0$, we obtain 
    \[
    \langle K_{\omega, \delta}^{\prime} 
(\phi_{\omega}), u_{*} \rangle + \frac{d h}{ d a}(0) 
\langle K_{\omega, \delta}^{\prime} 
(\phi_{\omega}), \phi_{\omega} \rangle = 0
    \]
Since $\langle K_{\omega, \delta}^{\prime} 
(\phi_{\omega}), \phi_{\omega} \rangle = - 2 (p-1) \int_{0}^{\infty} h_{\delta}(r) |\phi_{\omega}|^{p+1} dx < 0$, we see that $\frac{d h}{ d a}(0) = 0$. 
It follows that 
    \[
    \begin{split}
    \frac{d^{2} S_{\omega, \delta}}{d a^{2}}(Z(a, h(a)))
    & 
    = \frac{d}{d a} \left( 
    \langle S_{\omega, \delta}^{\prime}(Z(a, b)), 
    \phi_{\omega} + \frac{d h}{d a}(a) \phi_{\omega} \rangle
    \right) \\
    & = 
    \langle S_{\omega, \delta}^{\prime \prime}(Z(a, b))
    \left(\phi_{\omega} + \frac{d h}{d a}(a) \phi_{\omega} \right), 
    \phi_{\omega} + \frac{d h}{d a}(a) \phi_{\omega} \rangle \\
    & \quad 
    + \langle S_{\omega, \delta}^{\prime}(Z(a, b)), 
    \frac{d^{2} h}{d a^{2}}(a) \phi_{\omega} \rangle. 
    \end{split}
    \]
This together with $\frac{d h}{ d a}(0) = 0$
implies that 
    \[
    \frac{d^{2} S_{\omega, \delta}}{d a^{2}}(Z(a, h(a)))\Big|_{a = 0}
    = \langle S_{\omega, \delta}^{\prime \prime}(\phi_{\omega})
    \phi_{\omega}, \phi_{\omega} \rangle 
    = - 2(p-1) \int_{0}^{\infty} h_{\delta}(r) |\phi_{\omega}|^{p+1} dx < 0,  
    \]
which is a contradiction. 
This completes the proof. 
\end{proof}
From Lemma \ref{lem-morse1} and the mini-max theorem, 
we see that the Morse index of $S_{\omega, \delta}^{\prime \prime}(\phi_{\omega})$ 
is $1$.

\begin{thank}
N.F. was supported by JSPS KAKENHI Grant Number JP20K14349 and JP24H00024.
M.I was supported by JSPS KAKENHI Grant Number JP25H01453 and JP23K03174. H.K. was supported by JSPS KAKENHI 
Grant Number JP25H01453 and JP25K07089.
\end{thank}

 \subsection*{Data availability}
Data sharing not applicable to this article as no datasets were generated or analyzed during the current study.

\subsection*{Conflict of interest}
The authors declare that they have no conflict of interest.

\bibliographystyle{amsplain_abbrev_nobysame_nonumber} %
\bibliography{biblio}

\providecommand{\bysame}{\leavevmode\hbox to3em{\hrulefill}\thinspace}
\providecommand{\MR}{\relax\ifhmode\unskip\space\fi MR }
\providecommand{\MRhref}[2]{%
  \href{http://www.ams.org/mathscinet-getitem?mr=#1}{#2}
}
\providecommand{\href}[2]{#2}
\begin{thebibliography}{10}

\bibitem{MR3638314}
J.~Bellazzini, N.~Boussa\"id, L.~Jeanjean, and N.~Visciglia, \emph{Existence
  and stability of standing waves for supercritical {NLS} with a partial
  confinement}, Comm. Math. Phys. \textbf{353} (2017), 229--251.

\bibitem{MR646873}
H.~Berestycki and T.~Cazenave, \emph{Instabilit\'e{} des \'etats stationnaires
  dans les \'equations de {S}chr\"odinger et de {K}lein-{G}ordon non
  lin\'eaires}, C. R. Acad. Sci. Paris S\'er. I Math. \textbf{293} (1981),
  489--492.

\bibitem{BurcFero15}
A.~Burchard and A.~Ferone, \emph{On the extremals of the {P}\'olya-{S}zeg{\H o}
  inequality}, Indiana Univ. Math. J. \textbf{64} (2015), 1447--1463.

\bibitem{Carr_2000}
L.~D. Carr, C.~W. Clark, and W.~P. Reinhardt, \emph{Stationary solutions of the
  one-dimensional nonlinear schr^^c3^^b6dinger equation. i. case of repulsive
  nonlinearity}, Physical Review A \textbf{62} (2000).

\bibitem{MR677997}
T.~Cazenave and P.-L. Lions, \emph{Orbital stability of standing waves for some
  nonlinear {S}chr\"odinger equations}, Comm. Math. Phys. \textbf{85} (1982),
  549--561.

\bibitem{MR2002047}
T.~Cazenave, \emph{Semilinear {S}chr\"odinger equations}, Courant Lecture Notes
  in Mathematics, vol.~10, New York University, Courant Institute of
  Mathematical Sciences, New York; American Mathematical Society, Providence,
  RI, 2003.

\bibitem{MR1995870}
A.~Comech and D.~Pelinovsky, \emph{Purely nonlinear instability of standing
  waves with minimal energy}, Comm. Pure Appl. Math. \textbf{56} (2003),
  1565--1607.

\bibitem{MR4191499}
N.~Fukaya, \emph{Uniqueness and nondegeneracy of ground states for nonlinear
  {S}chr\"odinger equations with attractive inverse-power potential}, Commun.
  Pure Appl. Anal. \textbf{20} (2021), 121--143.

\bibitem{MR4020633}
N.~Fukaya and M.~Ohta, \emph{Strong instability of standing waves for nonlinear
  {S}chr\"odinger equations with attractive inverse power potential}, Osaka J.
  Math. \textbf{56} (2019), 713--726.

\bibitem{MR2123132}
R.~Fukuizumi, \emph{Stability of standing waves for nonlinear {S}chr\"odinger
  equations with critical power nonlinearity and potentials}, Adv. Differential
  Equations \textbf{10} (2005), 259--276.

\bibitem{MR1973275}
R.~Fukuizumi and M.~Ohta, \emph{Instability of standing waves for nonlinear
  {S}chr\"odinger equations with potentials}, Differential Integral Equations
  \textbf{16} (2003), 691--706.

\bibitem{MR1948875}
R.~Fukuizumi and M.~Ohta, \emph{Stability of standing waves for nonlinear
  {S}chr\"odinger equations with potentials}, Differential Integral Equations
  \textbf{16} (2003), 111--128.

\bibitem{MR2379460}
F.~Genoud and C.~A. Stuart, \emph{Schr\"odinger equations with a spatially
  decaying nonlinearity: existence and stability of standing waves}, Discrete
  Contin. Dyn. Syst. \textbf{21} (2008), 137--186.

\bibitem{MR901236}
M.~Grillakis, J.~Shatah, and W.~Strauss, \emph{Stability theory of solitary
  waves in the presence of symmetry. {I}}, J. Funct. Anal. \textbf{74} (1987),
  160--197.

\bibitem{LiebLoss2001}
E.~H. Lieb and M.~Loss, \emph{Analysis}, vol.~14, American Mathematical Soc.,
  2001.

\bibitem{MR2923422}
M.~Maeda, \emph{Stability of bound states of {H}amiltonian {PDE}s in the
  degenerate cases}, J. Funct. Anal. \textbf{263} (2012), 511--528.

\bibitem{MR1989539}
J.~B. McLeod, C.~A. Stuart, and W.~C. Troy, \emph{Stability of standing waves
  for some nonlinear {S}chr\"odinger equations}, Differential Integral
  Equations \textbf{16} (2003), 1025--1038.

\bibitem{MR3318740}
B.~Noris, H.~Tavares, and G.~Verzini, \emph{Existence and orbital stability of
  the ground states with prescribed mass for the {$L^2$}-critical and
  supercritical {NLS} on bounded domains}, Anal. PDE \textbf{7} (2014),
  1807--1838.

\bibitem{MR1016082}
Y.-G. Oh, \emph{Cauchy problem and {E}hrenfest's law of nonlinear
  {S}chr\"odinger equations with potentials}, J. Differential Equations
  \textbf{81} (1989), 255--274.

\bibitem{MR2785894}
M.~Ohta, \emph{Instability of bound states for abstract nonlinear
  {S}chr\"odinger equations}, J. Funct. Anal. \textbf{261} (2011), 90--110.

\bibitem{MR3842879}
M.~Ohta, \emph{Strong instability of standing waves for nonlinear
  {S}chr\"odinger equations with a partial confinement}, Commun. Pure Appl.
  Anal. \textbf{17} (2018), 1671--1680.

\bibitem{MR3792709}
M.~Ohta, \emph{Strong instability of standing waves for nonlinear
  {S}chr\"odinger equations with harmonic potential}, Funkcial. Ekvac.
  \textbf{61} (2018), 135--143.

\bibitem{MR493421}
M.~Reed and B.~Simon, \emph{Methods of modern mathematical physics. {IV}.
  {A}nalysis of operators}, Academic Press [Harcourt Brace Jovanovich,
  Publishers], New York-London, 1978.

\bibitem{MR939275}
H.~A. Rose and M.~I. Weinstein, \emph{On the bound states of the nonlinear
  {S}chr\"odinger equation with a linear potential}, Phys. D \textbf{30}
  (1988), 207--218.

\bibitem{MR804458}
J.~Shatah and W.~Strauss, \emph{Instability of nonlinear bound states}, Comm.
  Math. Phys. \textbf{100} (1985), 173--190.

\bibitem{ShioWata13}
N.~Shioji and K.~Watanabe, \emph{A generalized {P}oho\v zaev identity and
  uniqueness of positive radial solutions of {$\Delta u+g(r)u+h(r)u^p=0$}}, J.
  Differential Equations \textbf{255} (2013), 4448--4475.

\bibitem{ShioWata16}
N.~Shioji and K.~Watanabe, \emph{Uniqueness and nondegeneracy of positive
  radial solutions of {${\rm div}(\rho\nabla u)+\rho(-gu+hu^p)=0$}}, Calc. Var.
  Partial Differential Equations \textbf{55} (2016), Art. 32, 42.

\bibitem{MR2239285}
C.~A. Stuart, \emph{Uniqueness and stability of ground states for some
  nonlinear {S}chr\"odinger equations}, J. Eur. Math. Soc. (JEMS) \textbf{8}
  (2006), 399--414.

\bibitem{MR691044}
M.~I. Weinstein, \emph{Nonlinear {S}chr\"odinger equations and sharp
  interpolation estimates}, Comm. Math. Phys. \textbf{87} (1982/83), 567--576.

\end{thebibliography}

\vspace{0.5cm}

\noindent
Noriyoshi Fukaya
\\[6pt]
Waseda Research Institute for Science and Engineering, 
\\[6pt]
Waseda University,
\\[6pt]
Tokyo 169-8555, Japan
\\[6pt]
E-mail:nfukaya@aoni.waseda.jp

\vspace{0.5cm}

\noindent
Masahiro Ikeda
\\[6pt]
Graduate School of Information Science and Technology, 
\\[6pt]
Osaka University, 
\\[6pt]
1-5, Yamadaoka, Suita-shi, Osaka 565-0871, 
Japan
\\[6pt]
E-mail: ikeda@ist.osaka-u.ac.jp

\vspace{0.5cm}

\noindent
Hiroaki Kikuchi
\\[6pt]
Department of Mathematics,
\\[6pt]
Tsuda University,
\\[6pt]
2-1-1 Tsuda-machi, Kodaira-shi, Tokyo 187-8577, JAPAN
\\[6pt]
E-mail: hiroaki@tsuda.ac.jp

\end{document}